\documentclass[11pt,oneside]{article}

\usepackage{latexsym}
\usepackage{shortvrb}
\usepackage{graphicx}
\usepackage{fancyhdr}
\usepackage{amsfonts}
\usepackage{amsmath}
\usepackage{amssymb}
\usepackage{mathrsfs}
\usepackage{makeidx}
\usepackage[all]{xy}
\usepackage{amsthm}
\usepackage{tikz}

\usetikzlibrary{matrix}

\usepackage[T1]{fontenc}
\usepackage{amsfonts}
\usepackage{amsmath}

\newtheorem{thm}{Theorem}[section]
\newtheorem{lem}[thm]{Lemma}

\newtheorem{cor}[thm]{Corollary}

\theoremstyle{definition}

\newcommand{\blackged}{\hfill$\blacksquare$}
\newcommand{\whiteged}{\hfill$\square$}
\newcounter{proofcount}

\renewenvironment{proof}[1][\proofname.]{\par
  \ifnum \theproofcount>0 \pushQED{\whiteged} \else \pushQED{\blackged} \fi%
  \refstepcounter{proofcount}
  \normalfont 
  \trivlist
  \item[\hskip\labelsep
        \itshape
    {\bf\em #1}]\ignorespaces
}{%
  \addtocounter{proofcount}{-1}
  \popQED\endtrivlist
}

\begin{document}

\begin{center}
\textbf{Cylinder topological quantum field theory: A categorical presentation of classical field theory and its symmetries}
\end{center}

\begin{center}
\textit{Juan Orendain}
\end{center}

\noindent \textit{Abstract:} We use geometric ideas coming from certain classic algebraic constructions to associate, to every classical field theory, a symmetric monoidal double functor from the double category of cobordisms with corners to a certain symmetric monoidal double category. We call symmetric monoidal double functors so constructed cylinder topological quantum field theories. Our initial formulation of cylinder topological quantum field theory is presented in a purely topological context. We present appropriate refinements in order to accommodate information relevant to classical field theory.

\tableofcontents

\section{Introduction}

\noindent The subject matter of the present note is the cylinder topological quantum field theory construction. Roughly cylinder topological quantum field theory associates to every classical field theory a topological quantum field theory of the appropriate dimension. The purpose of the construction at present is twofold. On the one hand the cylinder topological quantum field theory construction provides examples of topological quantum field theories having topological interpretations of theories in classical physics as initial data and on the other hand it provides a novel categorical insight into classical field theory and its symmetries.

We use geometric ideas behind constructions in skein and braid theory [14][15][16], the theory of Temperly-Lieb-type objects [19][20], the classification theory of subfactors [21][22] and in the foundational theory of Walker and Morrison's blob complex [23] to offer categorical interpretations of topological models describing local classical field theory and their symmetries. Precisely, we use a special case of Walker's cylinder category construction [24] to associate to any classical field theory satisfying the axioms of the local framework introduced in [1] and [2], a symmetric monoidal double functor from the symmetric monoidal double category of oriented cobordisms of the appropriate dimension with codimension 1 corners to a certain symmetric monoidal double category. We call such double functors cylinder topological quantum field theories. The codomain double category of cylinder topological quantum field theories is the symmetric monoidal double category of fibered topological semi-groups and their fibered bimodules, which we introduce and study in detail. We present refinements of this double category in order to include the smooth, symplectic, and Lagrangian structures present in spaces of global solutions and spaces of germs of local solutions in classical field theory. The cylinder topological quantum field theory construction fits into a programme [1][26][27] whose aim is to provide a categorical understanding of Schrodinger-Feynman quantization beyond the usual approach of classical topological quantum field theory. The topological quantum field theory formalism employed in our presentation provides hints that such interpretations are possible. Further, such presentation fits into the classical theory of quantum invariants and offers, in principle, possible new interpretation of classical invariants of low dimensional manifolds via well understood classical field theories in low dimensions. We now sketch in detail the contents of this paper.

In section 2 we present the axiomatic formulation of classical field theory which will be used throughout the paper. We call theories within this framework local field theories. Our treatment of local field theory will follow [1]. We further present a version of local field theory in which space-time symmetries are taken into account. We call such theories equivariant local field theories. Equivariant local field theories will serve as initial data for cylinder topological quantum field theories.  In section 3 we introduce fibered topological semi-groups and fibered topological bimodules. We prove that fibered topological semigroups and their bimodules organize into an oplax symmetric monoidal double category. We introduce a condition on fibered topological semi-groups and their fibered bimodules making the double category generated fibered topological semi-groups satisfying this condition into a strict double category. We call this ocndition rigidity. The double category of rigid fibered topological semi-groups will serve as codomain double category for cylinder topological quantum field theory. In section 4 we introduce cylinder topological quantum field theor. We associate, through what we call the cylinder topological semi-group construction, a symmetric monoidal double functor from the double category of cornered cobordisms of appropriate dimension to the double category of rigid fibered topological semi-groups, its cylinder topological quantum field theory. Cylinder topological quantum field theory is the main contstruction of the paper. In sections 5 and 6 we introduce smooth and Lagrangian cylinder topological quantum field theory. We define the concepts of fibered Lie semi-group and of fibered symplectic semi-group. We associate to every equivariant local field theory a symmetric monoidal double functor from the double category of cornered cobordisms of the appropriate dimension to a symmetric monoidal category of fibered Lie semi-groups in the smooth case and to a symmetric monoidal double category of fibered symplectic semi-groups in the Lagrangian case. Smooth and Lagrangian cylinder topologcal quantum field theories are meant to account for the smooth, symplectic, and Lagrangian structures present in spaces of global solutions and spaces of germs of local solutions in local field theory.

\

\section{Equivariant local field theory}

\

\noindent In this first section we introduce equivariant local field theory. We divide this section into two parts. In the first part we present the definition of classical local field theory. In the second we present the definition of its equivariant counterpart. We explain how every classical local field theory can be considered as an equivariant local field theory.

\

\noindent \textit{Classical local field theory}

\

\noindent Before presenting the definition of classical local field theory we introduce the concept of gluing triple. We begin with the definition of closed gluing triple. Given codimension 0 closed submanifolds $\Sigma$ and $\Lambda$ of the boundary $\partial X$ of a manifold $X$ such that $\partial X$ contains a second copy $-\Sigma$ of $\Sigma$ with opposite orientation we will say that the triple $X,\Sigma,\Lambda$ forms an uncornered gluing triple if $\partial X$ admits a decomposition as the disjoint union

\[\partial X=\Lambda\sqcup \Sigma\sqcup -\Sigma\]

\noindent In the cornered case, that is in the case in which $\Sigma$ and $\Lambda$ have boundary, we will say that triple $X,\Sigma,\Lambda$ is a cornered gluing triple when equation $\partial \Lambda=(\partial \Sigma \sqcup \partial -\Sigma)\cap \partial \Lambda$ holds and when $\partial X$ admits a decomposition as the union

\[\partial X= \Lambda\cup \Sigma\cup -\Sigma\]

\noindent We will say that a gluing triple, that is a cornered or uncornered gluing triple $X,\Sigma,\Lambda$, is of dimension $n$ if manifold $X$ is of dimension $n$. Given a gluing triple $X,\Sigma,\Lambda$ we will write $X_{gl}^\Sigma$ or simply $X_{gl}$ when $\Sigma$ is clear from the context, for the manifold obtained from $X$ by gluing manifolds $\Sigma$ and $-\Sigma$. Observe that in this case the boundary $\partial X_{gl}$ of manifold $X_{gl}$ is equal to $\Lambda$ in the case in which $X,\Sigma,\Lambda$ is uncornered, and is equal to the manifold $\Lambda_{gl}$ obtained from the uncornered gluing triple $\Lambda,\partial\Sigma\cap\Lambda,\partial\Lambda$ in the case in which triple $X,\Sigma,\Lambda$ is cornered.

\

\noindent We now present the definition of classical local field theory. Let $n$ be  apositive integer. A local field theory of dimension $n$ consists of the following data:

\begin{enumerate}
\item \textbf{Germs of local solutions:} For each oriented manifold $\Sigma$ of dimension $n-1$ we assign a symplectic manifold $(L_\Sigma,\omega_\Sigma)$. We call $L_\Sigma$ the space of germs of local solutions on $L$.

\item \textbf{Global solutions:} To each oriented manifold $X$ of dimension $n$ we assign a smooth manifold $L_X$. We call $L_X$ the space of global solutions on $X$.

\item \textbf{Restriction transformation:} For each oriented manifold $X$ of dimension $n$ we associate a smooth function $r_X:L_X\to L_{\partial X}$ such that the image of $r_X$ is a Lagrangian submanifold of $L_{\partial X}$. We call $r_X$ the restriction function associated to $X$.

\noindent \textbf{Gluing transformation:} For each gluing triple $X,\Sigma,\Lambda$ where $X$ is of dimension $n$ we associate an embedding $\bullet_\Sigma:L_{X_{gl}}\to L_X$. We call $\bullet_\Sigma$ the gluing transformation associated to the triple $X,\Sigma,\Lambda$. We assume that the gluing transformation is associative.
\end{enumerate}

\noindent We further assume that the data described above satisfies the following conditions:

\begin{enumerate}
\item \textbf{Orientation sensitivity:} Given an oriented manifold $\Sigma$ of dimension $n-1$, the space of germs of local solutions $L_{-\Sigma}$ associated to $\Sigma$ with its opposite orientation is equal to space $L_\Sigma$ with symplectic form $-\omega_\Sigma$.

\item \textbf{Decomposition along hypersurfaces:} Given a decomposition of an oriented manifold $\Sigma$ of dimension $n-1$ as the union $\Sigma=\Sigma_1\cup_W\Sigma_2$ along a codimension 0 submanifold $W$ of the common boundary $\partial \Sigma_1\cap \partial \Sigma_2$ of $\Sigma_1$ and $\Sigma_2$ we assume that equation $L_\Sigma = L_{\Sigma_1}\times L_{\Sigma_2}$ holds.

\item \textbf{Diagonal axiom:} Given an oriented manifold $\Sigma$ of dimension $n-1$ we assume that the image $im\pi_{L_\Sigma^2}r_{\Sigma\times [0,1]}$ of the composition of the restriction transformation $r_{\Sigma\times [0,1]}$ associated to the cylinder $\Sigma\times [0,1]$ based on $\Sigma$ and the projection $\pi_{L_\Sigma^2}$ of $L_{\partial \Sigma\times [0,1]}$ onto $L_\Sigma^2$ is equal to the diagonal Lagrangian submanifold of $L_\Sigma\times -L_\Sigma$.

\item \textbf{Decomposition along regions:} Given a decomposition of an oriented manifold $X$ of dimension $n$ as the disjoint union $X=X_1\sqcup X_2$ of manifolds $X_1$ and $X_2$ we assume that equation $L_X=L_{X_1}\times L_{X_2}$ holds.

\item \textbf{Gluing:} Given a gluing triple $X,\Sigma,\Lambda$ of dimension $n$, the image of the gluing transformation $\bullet_\Sigma$ is the equalizer of compositions $\pi_{L_\Sigma}r_X$ and $\pi_{L_{-\Sigma}}r_X$ where $\pi_{L_\Sigma}$ and $\pi_{L_{-\Sigma}}$ are the canonical projections of $L_{\partial X}$ onto $L_\Sigma$ and $L_{-\Sigma}$. Finally, in this situation, we assume the following square commutes:

\begin{center}

\begin{tikzpicture}
  \matrix (m) [matrix of math nodes,row sep=3em,column sep=7em,minimum width=2em]
  {
     L_{X_{gl}}& L_X \\
     L_{\partial X} & L_\Lambda \\};
  \path[-stealth]
    (m-1-1) edge node [left] {$r_{X_{gl}}$} (m-2-1)
            edge node [above] {$\bullet_\Sigma$} (m-1-2)
    (m-1-2) edge node [right] {$r_X$} (m-2-2)
    (m-2-1) edge node [below] {$\pi_\Lambda$}(m-2-2);
\end{tikzpicture}
\end{center}

\end{enumerate}

\noindent We clarify the associativity condition imposed on the gluing transformation in the above definition. Given a decomposition of the boundary $\partial X$ of a an oriented manifold $X$ of dimension $n$ as the union

\[\partial X= \Lambda\cup \Sigma\cup -\Sigma\cup \Sigma'\cup-\Sigma'\]

\noindent such that both $X,\Sigma,\Lambda\cup \Sigma'\cup-\Sigma'$ and $X,\Sigma',\Lambda\cup \Sigma\cup-\Sigma$ are gluing triples such that both $X_{gl}^\Sigma,\Sigma',\Lambda_{gl}^\Sigma$ and $X_{gl}^{\Sigma'},\Sigma,W_{gl}^{\Sigma'}$ are also gluing triples, if we write $Y,Y'$ for the glued manifolds associated to these triples then associtivity of the gluing transformation $\bullet$ would vaguely mean that there is a coherent set of transformation between spaces $L_Y$ and $L_{Y'}$ obtained this way. For the sake of simplicity we will assume the existence of the strictest possible of such coherent set of transformations, that is, by the gluing transformation $\bullet_\Sigma$ associated to gluing triple $X,\Sigma,\Lambda$ being associative we will mean that in the situation presented above spaces $L_Y$ and $L_{Y'}$ are always equal. This choice is made out of convenience and we recognize that theories satisfying laxer associativity conditions might be equally if not richer than the ones considered here.

\

\noindent \textsl{Equivariant local field theory}

\

\

\noindent Before presenting the definition of equivariant local field theory we introduce the following notational conventions: Let $n$ be a positive integer. We will write $M^n$ for the groupoid of oriented manifolds of dimension $n$ and oriented diffeomorphisms. We will write $M^\infty$ for the groupoid of possibly infinte dimensional smooth manifolds and their diffeomorphisms. We will consider groupoids of the form $M^n$ and groupoid $M^\infty$ as being symmetric monoidal, with symmetric monoidal structure provided by disjoint union. We will further consider groupoids of the form $M^n$ as being involutive with involution provided by orientation reversal. We will denote by \textbf{Symp} the groupoid of (possibly infinite-dimensional) symplectic manifolds and their symplectomorphisms. We again consider \textbf{Symp} as an involutive symmetric monoidal groupoid, with symmetric monoidal structure now provided by cartesian product and involution by multiplying symplectic forms by -1. We will denote by $\partial_n$ or symply by $\partial$, for the symmetric monoidal functor from $M^n$ to $M^{n-1}$ associating to every manifold $X$ its boundary $\partial X$. Given a positive integer $n$ we will write $M^n_{gl}$ for the groupoid of gluing triples of dimension $n$ and oriented diffeomorphisms preserving gluing data. Thus defined $M^n_{gl}$ admits the structure of a symmetric monoidal groupoid, concrete over $M^n$. We will write $U_n$ or symply $U$ for the forgetful functor of $M^n_{gl}$ on $M^n$. We will write $\ast_n$ or symply $\ast$ for the functor from $M^n_{gl}$ to $M^n$ associating manifold $X_{gl}^\Sigma$ to every gluing triple $X,\Sigma,\Lambda$. Thus defined $\ast$ is again symmetric monoidal.

\

\noindent We now present the definition of equivariant local field theory. Given a positive integer $n$ an equivariant local field theory of dimension $n$ will consist of the following data: 

\begin{enumerate}

\item \textbf{Germs of local solutions:} A symmetric monoidal functor $L^{n-1}$ from $M^{n-1}$ to $\textbf{Symp}$. We call $L^{n-1}$ the functor of spaces of germs of local solutions of $L$. 

\item \textbf{Global solutions:} A symmetric monoidal functor $L^n$ from $M^n$ to $M^\infty$. We call $L^n$ the functor of spaces of global solutions of $L$.

\item \textbf{Restriction transformation:} A monoidal natural transformation $r$ from functor $L^n$ to composition $L^{n-1}\partial$. We assume that for every manifold $X$ of dimension $n$, the image of $r_X$ is a Lagrangian submanifold of $L_{\partial X}$. We call $r$ the restriction transformation of $L$.

\item \textbf{Gluing transformation:} A monoidal natural transformation $\bullet$ from $L^n\ast$ to $L^nU_n$. We call natural transformation $\bullet$ the gluing transformation associated to $L$.

\end{enumerate}

\noindent We assume that the data above is subject to the following conditions:

\begin{enumerate}

\item \textbf{Extended symmetry:} We assume that functor $L^{n-1}$ of spaces of germs of local solutions equalizes functors $U$ and $\ast$.

\item \textbf{Diagonal axiom:} Given an oriented manifold $\Sigma$ of dimension $n-1$ we assume that the image $im\pi_{L_\Sigma^2}r_{\Sigma\times [0,1]}$ of the composition of the restriction transformation $r_{\Sigma\times [0,1]}$ associated to the cylinder $\Sigma\times [0,1]$ based on $\Sigma$ and the projection $\pi_{L_\Sigma^2}$ of $L_{\partial \Sigma\times [0,1]}$ onto $L_\Sigma^2$ is equal to the diagonal Lagrangian submanifold of $L_\Sigma\times -L_\Sigma$.

\item \textbf{Gluing transformation:} We assume that the gluing transformation $\bullet$ is associative. Moreover, we assume that for every gluing triple $X,\Sigma,\Lambda$ of dimension $n$ the natural transformation $\bullet_\Sigma$ is an embedding with image equal to the equalizer of projections $\pi_{L_\Sigma}r_X$ and $\pi_{L_{-\Sigma}}r_X$ of $L_{\partial X}^{n-1}$ onto $ L^{n-1}_\Sigma$ and $L^{n-1}_{-\Sigma}$ and that the following square commutes:

\begin{center}

\begin{tikzpicture}
  \matrix (m) [matrix of math nodes,row sep=3em,column sep=7em,minimum width=2em]
  {
     L_{X_{gl}}& L_X \\
     L_{\partial X} & L_\Lambda \\};
  \path[-stealth]
    (m-1-1) edge node [left] {$r_{X_{gl}}$} (m-2-1)
            edge node [above] {$\bullet_\Sigma$} (m-1-2)
    (m-1-2) edge node [right] {$r_X$} (m-2-2)
    (m-2-1) edge node [below] {$\pi_\Lambda$}(m-2-2);
\end{tikzpicture}
\end{center}

\end{enumerate}

\

\noindent The axioms presented above are a straightforward categorification of the set of axioms defining local field theory. We wish for theories in our framework to satisfy an additional condition, which we now describe. Equivariant local field theories associate similar, i.e. isomorphic spaces of local or global solutions to similar, i.e. diffeomorphic manifolds. We require that equivariant theories satisfy a second order version of this, that is, we require for equivariant local field theories to associate similar solution space isomorphisms to similiar space-time diffeomorphisms. We clarify on what the word similar means in this context. Motivated by [W] we regard two diffeomorphisms between manifold $X,Y$ to be similar if they are isotopic relative to boundary. We regard two solution space isomorphisms to be similar when they are equal. Different choices of the notion of similarity between space-time or solution space isomorphisms will lead to different types of theories. This particular choice of type of theory suits our pourpuses. We now present the final axiom we require from equivariant field theory.

\

\noindent 4. \textbf{Re-parametrization invariance:} Let $\Sigma,\Sigma'$ be manifolds of dimension $n$ or $n-1$. Let $\Phi,\Phi'$ be diffeomorphisms from $\Sigma$ to $\Sigma'$. If $\Phi,\Phi'$ are isotopic relative to boundary then we assume that $L_\Phi$ and $L_{\Phi'}$ are equal.

\

\section{Fibered topological semi-groups}

\noindent In this section we introduce the symmetric monoidal weak double category of fibered topological semi-groups and their bimodules, which will serve as the target for weak topological quantum field theories associated to equivariant local field theories. We begin by presenting the category of objects of this double category.

\

\noindent \textsl{Category of objects}

\

\noindent We will understand for a fibered topological semi-group a semi-category internal to the category of topological spaces and continuous functions for which all its vertical morphisms are endomorphisms. Put another way, a fibered topological semi-group will be a quadruple $(E,X,\pi,\mu)$ where $E,X$ are topological spaces, and $\pi,\mu$ are continuous functions from $E$ to $X$ and from the fiber product $E\times_\pi E$ to $E$ respectively, such that functions $\pi$ and $\mu$ are compatible and function $\mu$ is associative, i.e. such that the following two squares commute

\begin{center}

\begin{tikzpicture}
  \matrix (m) [matrix of math nodes,row sep=3em,column sep=4em,minimum width=2em]
  {
     E\times_\pi E&E&E\times_\pi E\times_\pi E& E\times_\pi E \\
     E&X&E\times_\pi E&E \\};
  \path[-stealth]
    (m-1-1) edge node [left] {$p$}(m-2-1)
		        edge node [above]{$\mu$}(m-1-2)
		(m-1-2) edge node [right]{$\pi$}(m-2-2)
		(m-2-1) edge node [below]{$\pi$}(m-2-2)				
		
		(m-1-3) edge node [left] {$id_E\times \mu$} (m-2-3)
            edge node [above] {$\mu\times id_E$} (m-1-4)
    (m-1-4) edge node [right] {$\mu$} (m-2-4)
    (m-2-3) edge node [below] {$\mu$}(m-2-4);
\end{tikzpicture}
\end{center}

\noindent where $p$ in the first diagram denotes any of the two projections of $E\times_\pi E$ onto $E$. Given a fibered topological semi-group $(E,X,\pi,\mu)$ we will call $E,X,\pi$ and $\mu$ the total space, the base space, the projection, and the fibered product operation of $(E,X,\pi,\mu)$ respectively. We will usually identify fibered topological semi-groups with their corresponding total spaces. Given a fibered topological semi-group $E$ with base and projection $X$ and $\pi$ respectively and an element $x\in X$ the restriction, to the fiber $\pi^{-1}(x)$ of $x$, of the fibered product operation on $E$, provides $\pi^{-1}(x)$ with the structure of a topological semi-group.

\

\noindent The following are examples of fibered topological semi-groups. Let $E$ be a topological space. We can consider $E$ as a topological fibered semi-group by making $E$ its base, the identity $id_E$ of $E$ its projection, and any of the two projections of $E\times_{id_E}E$ onto $E$ as its fiber product operation. Let now $E$ be a topological semi-group with continuous product operation $\mu$. We now consider $E$ as a fibered topological semi-group by making the single point $\left\{\ast\right\}$ its base space, the constant function on $\ast$ the projection and operation $\mu$ as fiber product operation. Less trivial examples include vector bundles, with fiber-wise additive structure as fiber multiplication operation and principal bundles with fiber-wise product operation as fiber multiplication operation.

\

\noindent We will understand for a morphism of fibered topological semi-groups a semi-functor internal to the category of topological spaces and continuous functions. In other words, given fibered topological semi-groups $E$ and $E'$ with base spaces, projections, and fibered product operations $X,X'$, $\pi,\pi'$, and $\mu,\mu'$ respectively we will say that a pair $(\varphi,\psi)$ where $\varphi$ is a smooth function from $E$ to $E'$ and $\psi$ is a continuous function from $X$ to $X'$ is a morphism of fibered topological semi-groups from $E$ to $E'$ if $\varphi$ and $\psi$ intertwine projections and fibered products, that is, if the following two diagrams commute

\begin{center}

\begin{tikzpicture}
  \matrix (m) [matrix of math nodes,row sep=3em,column sep=6em,minimum width=2em]
  {
     E&E'&E\times_\pi E&E'\times_{\pi'} E' \\
     X&X'&E&E' \\};
  \path[-stealth]
    (m-1-1) edge node [left] {$\pi$} (m-2-1)
            edge node [above] {$\varphi$} (m-1-2)
    (m-1-2) edge node [right] {$\pi'$} (m-2-2)
    (m-2-1) edge node [below] {$\psi'$}(m-2-2)
		
		(m-1-3) edge node [left] {$\mu$}(m-2-3)
		(m-1-3) edge node [above]{$\varphi\times_\psi\varphi$}(m-1-4)
		(m-1-4) edge node [right]{$\mu'$}(m-2-4)
		(m-2-3) edge node [below]{$\varphi$}(m-2-4);
\end{tikzpicture}
\end{center}

\noindent We will usually identify morphisms of fibered topological semi-groups with their corresponding component in total spaces. Fibered topological semi-groups and their morphisms organize into a category. We denote this category by \textbf{fsGrp}$_0$. Cartesian product and order reversal on fibered product operations provide \textbf{fsGrp}$_0$ with the structure of an involutive symmetric monoidal category.

\

\noindent \textsl{Category of morphisms}

\

\noindent We will understand for a fibered bimodule over fibered topological semi-groups a bimodule over the corresponding semi-categories internal to topological spaces and continuous functions. Put another way, given fibered topological semi-groups $E$ and $E'$ with base spaces and projections $X,X'$ and $\pi,\pi'$ respectively, a left-right fibered bimodule over $E,E'$ will consist of a topological space $\Omega$, a pair of source and target continuous functions $s,t$ from $\Omega$ to $X$ and $X'$ respectively and associative left and right continuous fibered actions $\lambda:E\times_{\pi,s}\Omega\to\Omega$, $\rho:\Omega\times_{t,\pi'}E'\to \Omega$, compatible with source and target functions $s,t$ and compatible with themselves, that is such that the following diagrams commute

\begin{center}

\begin{tikzpicture}
  \matrix (m) [matrix of math nodes,row sep=3em,column sep=3em,minimum width=2em]
  {
    E\times_\pi E\times_{\pi,s}\Omega&E\times_{\pi,s}\Omega&\Omega\times_{t,\pi'}E'\times_{\pi'}E' &\Omega\times_{t,\pi'}E' \\
    E\times_{\pi,s}\Omega&\Omega&\Omega\times_{t,\pi'}E'&\Omega \\};
  \path[-stealth]
    (m-1-1) edge node [left] {$\mu\times id_\Omega$} (m-2-1)
            edge node [above] {$id_E\times\lambda$} (m-1-2)
    (m-1-2) edge node [right] {$\lambda$} (m-2-2)
    (m-2-1) edge node [below] {$\lambda$}(m-2-2)
		
		(m-1-3) edge node [left] {$id_\Omega\times\mu'$}(m-2-3)
		(m-1-3) edge node [above]{$\rho\times id_{E'}$}(m-1-4)
		(m-1-4) edge node [right]{$\rho$}(m-2-4)
		(m-2-3) edge node [below]{$\rho$}(m-2-4);
\end{tikzpicture}
\end{center}

\begin{center}

\begin{tikzpicture}
  \matrix (m) [matrix of math nodes,row sep=3em,column sep=6em,minimum width=2em]
  {
     E\times_{\pi,s}\Omega&\Omega&E\times_{\pi,s}\Omega&\Omega \\
     \Omega&X&\Omega&X' \\};
  \path[-stealth]
    (m-1-1) edge node [left] {$p_E$} (m-2-1)
            edge node [above] {$\lambda$} (m-1-2)
    (m-1-2) edge node [right] {$t$} (m-2-2)
    (m-2-1) edge node [below] {$\pi$}(m-2-2)
		
		(m-1-3) edge node [left] {$p_\Omega$}(m-2-3)
		(m-1-3) edge node [above]{$\lambda$}(m-1-4)
		(m-1-4) edge node [right]{$t$}(m-2-4)
		(m-2-3) edge node [below]{$t$}(m-2-4);
\end{tikzpicture}
\end{center}

\begin{center}

\begin{tikzpicture}
  \matrix (m) [matrix of math nodes,row sep=3em,column sep=6em,minimum width=2em]
  {
     \Omega\times_{t,\pi'} E'&\Omega&\Omega\times_{t,s'}E'&\Omega \\
     \Omega&X&E'&X' \\};
  \path[-stealth]
    (m-1-1) edge node [left] {$p_\Omega$} (m-2-1)
            edge node [above] {$\rho$} (m-1-2)
    (m-1-2) edge node [right] {$s$} (m-2-2)
    (m-2-1) edge node [below] {$s$}(m-2-2)
		
		(m-1-3) edge node [left] {$p_{E'}$}(m-2-3)
		(m-1-3) edge node [above]{$\rho$}(m-1-4)
		(m-1-4) edge node [right]{$t$}(m-2-4)
		(m-2-3) edge node [below]{$\pi'$}(m-2-4);
\end{tikzpicture}
\end{center}

\begin{center}

\begin{tikzpicture}
  \matrix (m) [matrix of math nodes,row sep=3em,column sep=8em,minimum width=2em]
  {
     E\times_{\pi,s}\Omega\times_{t,\pi'}E'& \Omega\times E' \\
     E\times_{\pi,s}\Omega&\Omega \\};
  \path[-stealth]
    (m-1-1) edge node [left] {$id_E\times\rho$} (m-2-1)
            edge node [above] {$\lambda\times id_{E'}$} (m-1-2)
    (m-1-2) edge node [right] {$\rho$} (m-2-2)
    (m-2-1) edge node [below] {$\lambda$}(m-2-2);
\end{tikzpicture}
\end{center}

\noindent Where $p_\Omega$ in the two diagrams of the second row above denotes the corresponding cartesian projection in both cases. We will usually identify fibered bimodules with their corresponding underlying spaces and will supress notation for left and right actions. We will write left and right actions of elements with matching fibers by concatenation. The notions of left and right fibered bimodule over fibered Lie semi-groups are defined analogously.


\

\noindent We will understand for an equivariant morphism of fibered bimodules an internal morphism of bimodules over semi-categories internal to the category of topological spaces and continuous functions. In other words, given fibered topological semi-groups $E,E',F$ and $F'$, and left-right fibered bimodules $\Omega$ and $\Omega'$ over $E,F$ and $E',F'$ respectively, we will say that a triple $(\varphi,\Phi,\psi)$ where $\varphi$ and $\psi$ are morphisms of fibered topological semi-groups from $E$ to $E'$ and from $F$ to $F'$ respectively, and where $\Phi$ is a continuous function from $\Omega$ to $\Omega'$, is an equivariant morphism from $\Omega$ to $\Omega'$ if the $\Pi,\varphi$ and $\psi$ are compatible with source and target functions and with left and right actions defining the fibered bimodule structure of $\Omega$ and $\Omega'$, that is if the following diagrams commute

\begin{center}

\begin{tikzpicture}
  \matrix (m) [matrix of math nodes,row sep=3em,column sep=7em,minimum width=2em]
  {
     \Omega&\Omega'&\Omega&\Omega' \\
     X&X'&Y&Y' \\};
  \path[-stealth]
    (m-1-1) edge node [left] {$s$} (m-2-1)
            edge node [above] {$\Phi$} (m-1-2)
    (m-1-2) edge node [right] {$s'$} (m-2-2)
    (m-2-1) edge node [below] {$\varphi$}(m-2-2)
		
		(m-1-3) edge node [left] {$t$}(m-2-3)
		(m-1-3) edge node [above]{$\Phi$}(m-1-4)
		(m-1-4) edge node [right]{$t'$}(m-2-4)
		(m-2-3) edge node [below]{$\psi$}(m-2-4);
\end{tikzpicture}
\end{center}

\begin{center}

\begin{tikzpicture}
  \matrix (m) [matrix of math nodes,row sep=3em,column sep=8em,minimum width=2em]
  {
     E\times_{\pi,s}\Omega& \Omega&\Omega\times_{t,\pi} F \\
     E'\times_{\pi',s'}\Omega'&\Omega'&\Omega'\times_{t',\pi'}F' \\};
  \path[-stealth]
    (m-1-1) edge node [left] {$\varphi\times\Phi$} (m-2-1)
            edge node [above] {$\lambda$} (m-1-2)
    (m-1-2) edge node [right] {$\Phi$} (m-2-2)
    (m-2-1) edge node [below] {$\lambda'$}(m-2-2)
		(m-1-3) edge node [above] {$\rho$}(m-1-2)
		(m-1-3) edge node [right] {$\Phi\times \psi$}(m-2-3)
		(m-2-3) edge node [below] {$\rho'$}(m-2-2);
\end{tikzpicture}
\end{center}

\noindent where $X$ and $X'$ denote the base spaces of $E$ and $E'$ respectively, $Y$ and $Y'$ denote the base spaces of $F$ and $F'$, $\pi$ and $\pi'$ denote the projections of $E$ and $F$ and of $E'$ and $F'$ respectively and where $\lambda,\lambda',\rho$ and $\rho'$ denote the left and right defining $\Omega$ and $\Omega'$ respectively.

\

\noindent Fibered bimodules over fibered topological semi-groups and their morphisms again organize into a category. We now denote this category by \textbf{fsGrp}$_1$. Cartesian product provides \textbf{fsGrp}$_1$ with the structure of a symmetric monoidal category.

\

\noindent \textsl{Structure of double category}

\

\noindent We now provide pair \textbf{fsGrp}$_0$,\textbf{fsGrp}$_1$ with the structure of a symmetric monoidal weak double category. We begin by defining the horizontal identity functor for \textbf{fsGrp}$_0$,\textbf{fsGrp}$_1$.

\

\noindent Given a fibered topological semi-group $E$ with base $X$ and projection $\pi$ we write $i_E$ for the following left-right fibered bimodule over $E$:
 We make the underlying space of $i_E$ to be $E$ itself, we make source and target functions on $i$ to be both equal to projection $\pi$ and we make left and right fibered actions of $E$ on $i_E$ to be those defined by fiber multiplication on the left and on the right respectively. Thus defined $i_E$ is a left-right fibered bimodule over $E$. Let now $\varphi$ be a morphism of fibered topological semi-groups from $E$ to another fibered topological semigroup $E'$. We will write in this case $i_\varphi$ for the triple $(\varphi,\varphi,\varphi)$. Thus defined $i_\varphi$ is an equivariant morphism of fibered bimodules from $i_E$ to $i_{E'}$. Finally, we write $i$ for the functor from \textbf{fsGrp}$_0$ to \textbf{fsGrp}$_1$ associating bimodule $i_E$ to every fibered topological semi-group $E$ and $i_\varphi$ to every morphism of topological semi-groups $\varphi$. Thus defined $i$ is clearly symmetric monoidal and will serve as the horizontal identity functor for pair  \textbf{fsGrp}$_0$,\textbf{fsGrp}$_1$. We now define the horizontal composition symmetric monoidal bifunctor for this pair.

\

\noindent Let $E,E'$ and $E''$ be fibered topological semi-groups with projections $\pi,\pi'$ and $\pi''$ respectively. Let $\Omega$ and $\Omega'$ be left-right fibered bimodules over $E,E'$ and over $E',E''$ respectively. Let $s,t,\lambda$ and $\rho$ be source and target functions and left and right actions for $\Omega$ respectively, and let $s',t',\lambda'$ and $\rho'$ be source and target functions and left and right actions for $\Omega'$ respectively. We write $\Omega\circledast_{E'}\Omega'$ for the product $\Omega\times_{t,s'}\Omega'$ fibered by functions $t$ and $s'$. We provide $\Omega\circledast_{E'}\Omega$ with the structure of a left-right fibered bimodule over $E$ and $E'$. We begin by defining source and target functions for $\Omega\circledast_{E'}\Omega'$. Let $p_\Omega$ and $p_{\Omega'}$ be the projections of $\Omega\circledast_{E'}\Omega'$ onto $\Omega$ and $\Omega'$ respectively. We make source transformation $\overline{s}$ on $\Omega\circledast_{E'}\Omega'$ to be the composition $sp_\Omega$ of $p_\Omega$ and source function $s$ on $\Omega$. Similarly we make target function $\overline{t}$ on $\Omega\circledast_{E'}\Omega'$ to be composition $tp_{\Omega'}$. We now define left and right actions $\overline{\lambda}$ and $\overline{\rho}$ of $E$ on $\Omega\circledast_{E'}\Omega'$ and of $E''$ on $\Omega\circledast_{E'}\Omega'$. Observe that by the compatibility of left action $\lambda$ of $E$ on $\Omega$ and target function $t$ on $\Omega$ function $\lambda\times id_{E'}$ restricts to a function from $E\times_{\pi,s}(\Omega\circledast_{E'}\Omega')$ to $\Omega\circledast_{E'}\Omega'$. We make action $\overline{\lambda}$ of $E$ on $\Omega\circledast_{E'}\Omega'$ to be this function. The right action $\overline{\rho}$ of $E''$ on $\Omega\circledast_{E'}\Omega'$ is defined similarly by right action $\rho$ of $E''$ on $\Omega'$. Compatibility of left action $\overline{\lambda}$ with source function $\overline{s}$ on $\Omega\circledast_{E'}\Omega$ follows from the compatibility of left action $\lambda$ of $E$ on $\Omega$ and source function $s$ of $\Omega$. Compatibility of $\overline{\lambda}$ with target function $\overline{t}$ on $\Omega\circledast_{E'}\Omega'$ follows directly from the corresponding definitions. An analogous argument proves compatibility of right action $\overline{\rho}$ with source and target $\overline{s},\overline{t}$ on $\Omega\circledast_{E'}\Omega'$. Associativity of $\overline{\lambda}$ and $\overline{\rho}$ follow from the associativity of the corresponding actions $\lambda$ and $\rho$. Finally, compatibility of left and right actions $\overline{\lambda}$ and $\overline{\rho}$ on $\Omega\circledast_{E'}\Omega'$ follows directly from their corresponding definitions.  We have proven the following lemma.

\begin{lem}
In the situation above functions $\overline{s},\overline{t},\overline{\lambda}$ and $\overline{\rho}$ provide space $\Omega\circledast_{E'}\Omega'$ with the structure of a left-right fibered bimodule over $E$ and $E''$.
\end{lem}

\noindent We call the bimodule in the above lemma the horizontal composition of fibered bimodules $\Omega$ and $\Omega'$ relative to fibered topological semi-group $E'$. We now extend the definition of horizontal composition of fibered bimodules to the horizontal composition of equivariant morphisms of fibered bimodules. Let $i\in\left\{1,2\right\}$. Let $E_i,E'_i$ and $E''_i$ be fibered topological semi-groups. Let $\Omega_i,\Omega'_i$ be left-right fibered bimodules over $E_i,E'_i$ and over $E'_i,E''_i$ respectively. Let $(\varphi,\Phi,\psi)$ and $(\psi,\Psi,\phi)$ be equivariant morphisms from $\Omega_1$ to $\Omega_2$ and from $\Omega'_1$ to $\Omega'_2$ respectively. In that case from equivariance of pair $(\Phi,\psi)$ with respect to the right action $\rho_1$ of $E'_1$ on $\Omega_1$ and the right action $\rho_2$ of $E'_2$ on $\Omega_2$  and from equivariance of pair $(\psi,\Psi)$ with respect to left action $\lambda'_1$ of $E'_1$ on $\Omega'_1$ and left action $\lambda'_2$ of $E'_2$ on $\Omega'_2$ it follows that the cartesian product $\Phi\times\Psi$ restricts to a continuous function from $\Omega_1\circledast_{E'_1}\Omega'_1$ to $\Omega_2\circledast_{E'_2}\Omega'_2$. We write $\Phi\circledast_\psi\Psi$ for this function. Now, equivariance of the pair $(\varphi,\Phi)$ with respect to left action $\lambda_1$ of $E_1$ on $\Omega_1$ and left action $\lambda_2$ of $E_2$ on $\Omega_2$ together with equivariance of the pair $(\Psi,\phi)$  with respect to the right action $\rho'_1$ of $E''_1$ on $\Omega'_1$ and right action $\rho'_2$ of $E''_2$ on $\Omega'_2$ it follows that triple $(\varphi,\Phi\circledast_\psi\Psi,\phi)$ is an equivariant morphism of fibered bimodules from $\Omega_1\circledast_{E'_1}\Omega'_1$ to $\Omega_2\circledast_{E'_2}\Omega'_2$. We call triple $(\varphi,\Phi\circledast_\psi\Psi,\phi)$ the horizontal composition of $(\varphi,\Phi,\psi)$ and $(\psi,\Psi,\phi)$. We now prove the followin theorem.

\begin{thm}
Horizontal identity funtor $i$ and horizontal composition bifunctor $\circledast$ provide pair \textbf{fsGrp}$_0$,\textbf{fsGrp}$_1$ with the structure of a weak symmetric monoidal double category, with strictly associative horizontal composition and lax identity transformations.
\end{thm}

\begin{proof}
We wish to prove that horizontal identity functor $i$ and horizontal composition bifunctor $\circledast$ provide the pair \textbf{fsGrp}$_0$,\textbf{fsGrp}$_1$ with the structure of a symmetric monoidal double category.

We make source and target functors for pair \textbf{fsGrp}$_0$,\textbf{fsGrp}$_1$ to be the obvious choices. These functors are clearly symmetric monoidal and are compatible with horizontal identity functor $i$ and horizontal composition bifunctor $\circledast$.

We now construct left and right horizontal identity transformations $L$ and $R$. Let $E$ and $E'$ be fibered topological semi-groups. Let $\Omega$ be a left-right fibered bimodule over $E$ and $E'$ with left and right actions $\lambda$ and $\rho$. Write $L^\Omega$ for the triple $(id_E,\lambda,id_{E'})$ and write $R^\Omega$ for the triple $(id_e,\rho,id_{E'})$. Thus defined $L^\Omega$ and $R^\Omega$ are equivariant morphisms of fibered bimodules from $E\circledast_E\Omega$ and from $\Omega\circledast_{E'}E'$ respectively to $\Omega$. Left equivariance of $L^\Omega$ and right equivariance of $R^\Omega$ follow from associativity of left action $\lambda$ of $E$ on $\Omega$ and from right action $\rho$ of $E'$ on $\Omega$ respectively. Right equivariance of $L^\Omega$ and left equivariance of $R^\Omega$ both follow directly from the corresponding definitions. Write now $L$ and $R$ for the collection of all equivariant morphisms of the form $L^\Omega$ and for the collection of all equivariant morphisms of the form $R^\Omega$ respectively. Thus defined $L$ and $R$ are natural transformations. This follows from the definition of equivariance of morphisms of fibered bimodules and from the way horizontal composition of equivariant morphisms of fibered bimodules was constructed. Moreover, $L$ and $R$ are clearly monoidal. 

Strict associativity of horizontal composition bifunctor $\circledast$ follows from associativity of the operation of taking fibered products of topological spaces. Finally, the triangle equations for left and right identity transformations $L$ and $R$ follow directly from the compatibility condition for left and right actions on fibered bimodules. This concludes the proof. 
\end{proof}

\noindent We will write \textbf{fsGrp} for the double category described in the theorem above. We now improve on the condition of left and right identity transformations for \textbf{fsGrp} being lax by considering a sub-double category of \textbf{fsGrp} having strong identity transformations as target double category for topological quantum field theories associated to equivariant local field theories.

\

\noindent Let $E$ be a topological fibered semi-group with projection $\pi$ and fiber product operation $\mu$. We will say that $E$ is a rigid fibered semi-group if $\mu$ defines a homeomorphism between $E\times_\pi E$ and $E$. Observe that considered as fibered topological semi-groups a vector bundle is rigid only when fibers are of dimension 0 and a principal bundle is rigid only when the group action is trivial. We write $\tilde{\mbox{\textbf{fsGrp}}}_0$ for the full symmetric monoidal subcategory of the category of objects \textbf{fsGrp}$_0$ of \textbf{fsGrp} generated by the collection of all rigid fibered topological semi-groups. Let now $E$ and $E'$ be rigid fibered topological semi-groups and let $\Omega$ be a left-right fibered bimodule over $E$ and $E'$ with source and target functions $s$ and $t$ and left and right actions $\lambda$ and $\rho$. We will say that in this case $\Omega$ is a rigid fibered bimodule if $\lambda$ is a homeomorphism between $E\times_{\pi,s}\Omega$ and $\Omega$ and $\rho$ is a homeomorphism between $\Omega\times_{t,\pi}E'$ and $\Omega$. We now write $\tilde{\mbox{\textbf{fsGrp}}}_1$ for the full symmetric monoidal subcategory of the category of morphisms \textbf{fsCat}$_1$ of \textbf{fsCat} generated by rigid fibered bimodules. It is easily seen that the horizontal identity $i_E$ of a fibered topological semi-group $E$ is a rigid fibered bimodule if and only if the fibered semi-group $E$ is itself rigid and that the horizontal composition $\Omega\circledast\Omega'$ of two fibered bimodules $\Omega$ and $\Omega'$ over rigid fibered semi-groups is rigid when each of $\Omega$ and $\Omega'$ is rigid. The horizontal identity functor $i$ and the horizontal composition bifunctor $\circledast$ on \textbf{fsCat} thus restrict to a symmetric monoidal functor from $\tilde{\mbox{\textbf{fsGrp}}}_0$ to $\tilde{\mbox{\textbf{fsGrp}}}_1$ and to a symmetric monoidal bifunctor from $\tilde{\mbox{\textbf{fsGrp}}}_1\times_{\tilde{\mbox{\textbf{fsGrp}}}_0}\tilde{\mbox{\textbf{fsGrp}}}_1$ to $\tilde{\mbox{\textbf{fsGrp}}}_0$. With this structure $\tilde{\mbox{\textbf{fsGrp}}}$ is thus a symmetric monoidal sub-double category of \textbf{fsGrp} for which by the rigidity condition left and irght horizontal identity transformations are isomorphisms. We will make $\tilde{\mbox{\textbf{fsGrp}}}$ the target double category for topological quantum field theories associated to equivariant field theories in the next section.

\section{Cylinder topological quantum field theories}

\noindent In this section we present the cylinder cornered topological quantum field theory construction. We explain how to associate, to every equivariant local field theory, a symmetric monoidal double functor from the symmetric monoidal double category of cornered cobordisms of the corresponding dimension to the symmetric monoidal double category of rigid fibered topological semi-groups defined in the previous section.

\

\noindent \textsl{Functor of objects}

\

\noindent We begin by defining functors of objects of cornered topological quantum field theories associated to equivariant local field theories. We associate, to every equivariant local field theory, a symmetric monoidal functor from the category of objects of the symmetric monoidal double category of cornered cobordisms of the corresponding dimension, to the symmetric monoidal category of rigid fibered topological semi-groups. We begin by defining this functor on objects.

\

\noindent Let $n$ be a positive integer. Let $L$ be an equivariant local field theory of dimension $n$. We explain how $L$ associates to every oriented manifold of dimension $n-1$ a fibered topological semi-group, which we will denote by $E^L_\Sigma$. Let $\Sigma$ be an oriented manifold of dimension $n-1$. We will make the total space of fibered topological semi-group $E^L_\Sigma$ to be the topological space underlying the space $L_{\Sigma\times [0,1]}$ of global solutions on the cylinder $\Sigma\times [0,1]$ based on $\Sigma$. We will make the base space of fibered topological semi-group $E^L_\Sigma$ to be the topological space underlying the space of germs of local solutions $L_\Sigma$ on manifold $\Sigma$. By the diagonal axiom in the definition of local field theory the image of the composition $p_{L_\Sigma^2}r_{\Sigma\times [0,1]}$ of restriction transformation $r_{\Sigma\times [0,1]}$ associated to $\Sigma\times [0,1]$ and the projection of the space of germs of local solutions $L_{\partial (\Sigma\times [0,1])}$ of the boundary $\partial (\Sigma\times [0,1])$ of $\Sigma\times [0,1]$ on the space of germs of local solutions $L_\Sigma^2$ on the boundary piece corresponding to the top and bottom copies of $\Sigma$. We will make the projection $\pi$ of $E^L_\Sigma$ to be the composition of $p_{L_\Sigma^2}r_{\Sigma\times [0,1]}$ and any of the left or right projections of $L_\Sigma^2$ onto $L_\Sigma$. In order to define the fiber product operation $\mu$ on $E^L_\Sigma$ now consider a diffeomorphism $\Phi$ from the glued cylinder $(\Sigma\times[0,1])\cup_\Sigma(\Sigma\times [0,1])$ to the cylinder $\Sigma\times [0,1]$ respecting fibers and covering the identity diffeomorphism on the bottom and top copies of $\Sigma$ on both manifolds, that is, making the following diagram commute

\begin{center}

\begin{tikzpicture}
  \matrix (m) [matrix of math nodes,row sep=4em,column sep=7em,minimum width=2em]
  {
     \Sigma\times[0,1]\cup_\Sigma\Sigma\times [0,1]&\Sigma\times [0,1] \\
     \Sigma&\Sigma \\};
  \path[-stealth]
    (m-1-1) edge node [above] {$\Phi$} (m-1-2)
            edge node [left] {$p_\Sigma$} (m-2-1)
    (m-1-2) edge node [right] {$p_\Sigma$} (m-2-2)
		(m-2-1) edge node [below] {$id_\Sigma$} (m-2-2);
\end{tikzpicture}
\end{center}

\noindent where $p_\Sigma$ denotes either one of the top or bottom projections of both $(\Sigma\times [0,1])\cup_\Sigma(\Sigma\times [0,1])$ and $\Sigma\times [0,1]$ onto $\Sigma$. We now define fiber product operation $\mu$ on $E^L_\Sigma$ through the following commutative diagram

\begin{center}

\begin{tikzpicture}
  \matrix (m) [matrix of math nodes,row sep=4em,column sep=7em,minimum width=2em]
  {
     E^L_\Sigma\times_\pi E^L_\Sigma&E^L_\Sigma \\
     E^L_\Sigma& \\};
  \path[-stealth]
    (m-1-1) edge node [above] {$\mu$} (m-1-2)
            edge node [left] {$\bullet^{-1}_\Sigma$} (m-2-1)
    (m-2-1) edge node [right] {$L_\Phi$} (m-1-2);
\end{tikzpicture}
\end{center}

\noindent where $\bullet_\Sigma$ denotes the gluing transformation associated by $L$ to the gluing triple formed by the disjoint union of the two copies of the cylinder $\Sigma\times [0,1]$ and the top and bottom copies of $\Sigma$ in the boundary $\partial (\Sigma\times [0,1])$ of $\Sigma\times [0,1]$. We prove the following lemma.

\begin{lem}
In the situation described above the fiber product operation $\mu$ defined on $E^L_\Sigma$ is well defined and provides $E^L_\Sigma$ with the structure of a rigid fibered topological semi-group.
\end{lem}

\begin{proof}
Let $n$ be a positive integer. Let $L$ be an equivariant local field theory of dimension $n$. Let $\Sigma$ be an oriented manifold of dimension $n-1$. We wish to prove that fiber product operation $\mu$ defined in the paragraph above does not depend on the specific choice of diffeomorphism $\Phi$ and that thus defined $\mu$ provides $E^L_\Sigma$ with the structure of a rigid fibered topological semi-group. 

Let $\Psi$ be another diffeomorphism from $(\Sigma\times [0,1])\cup_\Sigma (\Sigma\times [0,1])$ to the cylinder $\Sigma\times [0,1]$, preserving fibers and covering the identity diffeomorphism of $\Sigma$. The condition that both $\Phi$ and $\Psi$ preserve fibers and cover the identity diffeomorphism of $\Sigma$ implies that $\Phi$ and $\Psi$ are isotopically equivalent relative to boundary [wm] and thus by the axiom of re-parametrization invariance in the definition of $L$ the images $L_\Phi$ and $L_\Psi$ of $\Phi$ and $\Psi$ under the functor of global solutions of $L$ are equal. The commutative diagrams defining fiber product operation $\mu$ on $E^L_\Sigma$ relative to both $\Phi$ and $\Psi$ are thus equal. We conclude that the definition of $\mu$ is independent of the specific choice of diffeomorphism $\Phi$ satisfying the prescribed conditions.

We now prove that thus defined $\mu$ provides $E^L_\Sigma$ with the structure of a rigid fibered topological semi-group. We begin by rpoving the ocmpatibility of fiber product operation $\mu$ and projection $\pi$. Consider the following diagram

\begin{center}

\begin{tikzpicture}
  \matrix (m) [matrix of math nodes,row sep=4em,column sep=3em,minimum width=2em]
  {
     L_{\Sigma\times [0,1]}\times _{L_\Sigma}L_{\Sigma\times [0,1]}&L_{(\Sigma\times [0,1])\cup_\Sigma(\Sigma\times [0,1])}&L_{\Sigma\times [0,1]}\\
     L_{\partial(\Sigma\times [0,1])}^2&L_{\partial[(\Sigma\times [0,1])\cup_\Sigma(\Sigma\times [0,1])]}&L_{\partial(\Sigma\times [0,1])} \\};
  \path[-stealth]
    (m-1-1) edge node [above] {$\bullet^{-1}_\Sigma$} (m-1-2)
            edge node [left] {$r^2$} (m-2-1)
		(m-1-2) edge node [left] {$r$} (m-2-2)
		(m-2-1) edge node [below] {$id$} (m-2-2)
		(m-1-2) edge node [above] {$L_\Phi$} (m-1-3)
		(m-1-3) edge node [right] {$r$} (m-2-3)
		(m-2-2) edge node [below] {$id$} (m-2-3);
\end{tikzpicture}
\end{center}

\noindent Monoidality of restriction transformation $r$ together with compatibility of $r$ and gluing transformation $\bullet$ implies that the square on the left-hand side of the above diagram is commutative. Commutativity of the square on the right-hand side of the diagram follows directly from naturality of restriction transformation $r$ and the fact that diffeomorphism $\Phi$ covers the identity diffeomorphism of $\Sigma$. The above diagram is thus commutative. Compatibility of fiber product operation $\mu$ and projection $\pi$ on $E^L_\Sigma$ follows from this and from extended monoidality of the functor of germs of local solutions of $L$.

We now prove associativity of fiber product operation $\mu$. Consider now the follwoing diagram.

\begin{center}

\begin{tikzpicture}
  \matrix (m) [matrix of math nodes,row sep=4em,column sep=6em,minimum width=2em]
  {
     L_{(\Sigma\times [0,1])\cup_\Sigma (\Sigma\times [0,1])\cup_\Sigma (\Sigma\times [0,1]) }&L_{(\Sigma\times [0,1])\cup_\Sigma (\Sigma\times [0,1])}\\
     L_{(\Sigma\times [0,1])\cup_\Sigma (\Sigma\times [0,1])}&L_{\Sigma\times [0,1]}\\};
  \path[-stealth]
    (m-1-1) edge node [above] {$L_{\Phi\times id_{\Sigma\times [0,1]}}$} (m-1-2)
            edge node [left] {$L_{id_{\Sigma\times [0,1]}\times \Phi}$} (m-2-1)
		(m-1-2) edge node [left] {$L_\Phi$} (m-2-2)
		(m-2-1) edge node [below] {$L_\Phi$} (m-2-2);
\end{tikzpicture}
\end{center}

\noindent The commutativity of the above diagram follows from the fact that diffeomorphisms $\Phi(\Phi\times id_{\Sigma\times [0,1]})$ and $\Phi(id_{\Sigma\times [0,1]}\times \Phi)$ both preserve fibers and cover the identity diffeomorphism of the top and bottom copies of $\Sigma$ and are thus isotopically equivalent relative to boundary. Associativity of the fiber product operation follows directly from this together with the fact that restriction transformation $r$ is natural and monoidal. This proves that fiber product operation $\mu$ defined on $E^L_\Sigma$ the structure of a fibered topological semi-group. The fact that $\mu$ is a homeomorphism follows directly from the gluing transformation axiom in the definition of $L$. With this structure $E^L_\Sigma$ is thus a rigid fibered topological semigroup. This concludes the proof.
\end{proof}

\noindent We call fibered topological semi-group associated to $\Sigma$ above the cylinder fibered topological semi-group associated to $\Sigma$ by $L$. The following lemma extends the cylinder fibered topological semi-group construction to oriented diffeomorphisms.

\begin{lem}
Let $n$ be a positive integer. Let $L$ be an equivariant local field theory of dimension $n$. Let $\Sigma$ and $\Sigma'$ be oriented manifolds of dimension $n-1$. Let $\varphi$ be an oriented diffeomorphism from $\Sigma$ to $\Sigma'$. In that case the pair $(L_{\varphi\times id_{[0,1]}},L_\varphi)$ is an isomorphism of fibered topological semi-groups from the cylinder fibered topological semi-group $E^L_\Sigma$ associated to $\Sigma$ to the cylinder fibered topological semi-group $E^L_{\Sigma'}$ associated to $\Sigma'$.
\end{lem}

\begin{proof}
Let $n$ be a positive integer. Let $L$ be an equivariant local field theory of dimension $n$. Let $\Sigma$ and $\Sigma'$ be oriented manifolds of dimension $n-1$. Let $\varphi$ be an oriented diffeomorphism from $\Sigma$ to $\Sigma'$. We wish to prove that in that case pair $(L_{\varphi\times id_{[0,1]}},L_\varphi)$ is an isomorphism of fibered topological semi-groups from the cylinder fibered topological semi-group $E^L_\Sigma$ associated to $\Sigma$ to the cylinder fibered topological semi-group $E^L_{\Sigma'}$ associated to $\Sigma'$.

We begin by proving that pair $(L_{\varphi\times id_{[0,1]}},L_\varphi)$ is compatible with projections $\pi$ and $\pi'$ of $E^L_\Sigma$ and $E^L_{\Sigma'}$. Consider the following diagram

\begin{center}

\begin{tikzpicture}
  \matrix (m) [matrix of math nodes,row sep=3em,column sep=8em,minimum width=2em]
  {
     L_{\Sigma\times [0,1]}&L_{\Sigma'\times [0,1]} \\
     L_{\partial(\Sigma\times [0,1])}&L_{\partial(\Sigma'\times [0,1])} \\
		L_\Sigma^2&L_{\Sigma'}^2\\};
  \path[-stealth]
    (m-1-1) edge node [above] {$L_{\varphi\times id_{[0,1]}}$} (m-1-2)
            edge node [left] {$r_{\Sigma\times [0,1]}$} (m-2-1)
    (m-1-2) edge node [right] {$r_{\Sigma'\times [0,1]}$} (m-2-2)
		(m-2-1) edge node [below] {$L_{\tilde{\varphi\times id_{[0,1]}}}$} (m-2-2)
		(m-2-1) edge node [left] {$p_{L_\Sigma}^2$} (m-3-1)
		(m-2-2) edge node [right] {$p_{L_{\Sigma'}}^2$} (m-3-2)
		(m-3-1) edge node [below] {$L_\varphi^2$} (m-3-2);
\end{tikzpicture}
\end{center}

\noindent where $\tilde{\varphi\times id_{[0,1]}}$ denotes the restriction of $\varphi\times id_{[0,1]}$ to the boundary $\partial(\Sigma\times [0,1])$ of $\Sigma\times [0,1]$. The square at the top of the diagram is commutative by naturality of restriction transformation $r$ while the commutativity of the square at the bottom is obvious. The diagram is thus commutative. The commutativity of this diagram implies the compatibility of the pair $(L_{\varphi\times id_{[0,1]}},L_\varphi)$ with projections $\pi$ and $\pi'$ of $E^L_\Sigma$ and $E^L_{\Sigma'}$ respectively. We now prove that pair $(L_{\varphi\times id_{[0,1]}},L_\varphi)$ is compatible with fiber product operations $\mu$ and $\mu'$ on $E^L_\Sigma$ and $E^L_{\Sigma'}$ respectively. 

Let $\Phi$ be a collar diffeomprhism from the double cylinder $(\Sigma\times [0,1])\cup_\Sigma(\Sigma\times [0,1])$ to the cylinder $\Sigma\times [0,1]$ relative to a collar neighborhood of $\Sigma$ in $(\Sigma\times [0,1])\cup_\Sigma(\Sigma\times [0,1])$. Let $\Psi$ be a collar diffeomprhism from the double cylinder $(\Sigma'\times [0,1])\cup_{\Sigma'}(\Sigma'\times [0,1])$ to the cylinder $\Sigma'\times [0,1]$ relative to a collar neighborhood of $\Sigma'$ in $(\Sigma'\times [0,1])\cup_{\Sigma'}(\Sigma'\times [0,1])$ making the following diagram commute

\begin{center}

\begin{tikzpicture}
  \matrix (m) [matrix of math nodes,row sep=3em,column sep=9em,minimum width=2em]
  {
     (\Sigma\times [0,1])\cup_\Sigma(\Sigma\times [0,1])&(\Sigma'\times [0,1])\cup_{\Sigma'}(\Sigma'\times [0,1]) \\
     \Sigma\times [0,1]&\Sigma'\times [0,1] \\};
  \path[-stealth]
    (m-1-1) edge node [above] {$(\varphi\times id_{[0,1]})\cup_\varphi(\varphi\times id_{[0,1]})$} (m-1-2)
            edge node [left] {$\Phi$} (m-2-1)
    (m-1-2) edge node [right] {$\Psi$} (m-2-2)
		(m-2-1) edge node [below] {$\varphi\times id_{[0,1]}$} (m-2-2);
\end{tikzpicture}
\end{center}

\noindent Thus defined $\Psi$ clearly preserves fibers and dominates the identity diffeomorphism on the top and bottom copies of $\Sigma$ in $(\Sigma\times [0,1])\cup_\Sigma(\Sigma\times [0,1])$ and $\Sigma\times [0,1]$. From the diagram above we obtain commutativity of the botoom square of diagram

\begin{center}

\begin{tikzpicture}
  \matrix (m) [matrix of math nodes,row sep=3em,column sep=8em,minimum width=2em]
  {
     L_{\Sigma\times [0,1]}\times_{L_\Sigma}&L_{\Sigma'\times [0,1]}\times_{L_{\Sigma'}}L_{\Sigma'\times [0,1]} \\
     L_{(\Sigma\times [0,1])\cup_\Sigma(\Sigma\times [0,1])}&L_{(\Sigma'\times [0,1])\cup_{\Sigma'}(\Sigma'\times [0,1])} \\
		L_{\Sigma\times [0,1]}&L_{\Sigma'\times [0,1]}\\};
  \path[-stealth]
    (m-1-1) edge node [above] {$L_{\varphi\times id_{[0,1]}}\times L_{\varphi\times id_{[0,1]}}$} (m-1-2)
            edge node [left] {$\bullet_\Sigma^{-1}$} (m-2-1)
    (m-1-2) edge node [right] {$\bullet_{\Sigma'}^{-1}$} (m-2-2)
		(m-2-1) edge node [below] {$L_{(\varphi\times id_{[0,1]})\cup_\varphi(\varphi\times id_{[0,1]})}$} (m-2-2)
		(m-2-1) edge node [left] {$L_\Phi$} (m-3-1)
		(m-2-2) edge node [right] {$L_{\Phi'}$} (m-3-2)
		(m-3-1) edge node [below] {$L_{\varphi\times id_{[0,1]}}$} (m-3-2);
\end{tikzpicture}
\end{center}

\noindent The commutativity of the top square of the above diagram follows from naturality of gluing transformation $\bullet$. The diagram above is thus commutative. The commutativity of this diagram proves compatibility of pair $(L_{\varphi\times id_{[0,1]}},L_\varphi)$ with fiber product operation $\mu$. Pair 
$(L_{\varphi\times id_{[0,1]}},L_\varphi)$ is thus an isomorphism of fibered topological semi-groups from $E^L_\Sigma$ to $E^L_{\Sigma'}$. This concludes the proof.
\end{proof}

\noindent In the situation of the lemma above we will write $E^L_\varphi$ for the pair $(L_{\varphi\times id_{[0,1]}},L_\varphi)$ associated to oriented diffeomorphism $\varphi$. We will call isomorphism $E^L_\varphi$ the cylinder isomorphism associated to $\varphi$ by $L$. Given an equivariant local field theory $L$ of dimension $n$ we write $E^L_0$ for the functor from the groupoid $\tilde{\textbf{Cob}(n)}_0$ of oriented manifolds of dimension $n-1$ and their diffeomorphisms to the category $\tilde{\mbox{\textbf{fsGrp}}}_0$ of fibered topological semi-groups associating to every oriented manifold $\Sigma$ of dimension $n-1$ its cylinder fibered topological semi-group $E^L_\Sigma$ and associating cylinder isomorphism $E^L_\varphi$ to every oriented diffeomorphism $\varphi$. We prove the following lemma.

\begin{lem}
Let $n$ be a positive integer. Let $L$ be an equivariant local field theory of dimension $n$. In that case $E^L_0$ as defined above is symmetric monoidal and involutive.
\end{lem}

\begin{proof}
Let $n$ be a positive integer. Let $L$ be an equivariant local field theory of dimension $n$. We wish to prove that functor $E_0^L$ defined in the paragraph above is symmetric monoidal and involutive.

Symmetric monoidality of $E^L_0$ follows directly from symmetric monoidality of both the functor of spaces of germs of local solutions of $L$ and the functor of global solutions of $L$ together with monoidality of restriction transformation $r$ of $L$. We We prove that $E^L_0$ is involutive. Let $\Sigma$ be an oriented manifold of dimension $n-1$. We prove that fibered topological semi-groups $E^L_{-\Sigma}$ and $E^{L op}_\Sigma$ are equal, where $-\Sigma$ denotes manifold $\Sigma$ with its opposite orientation and $E^{L op}_\Sigma$ denotes fibered topological semi-group $E^L_\Sigma$ with the order of fiber product operation reversed. Total spaces, base spaces, and projections of both $E^L_{-\Sigma}$ and $E^{L op}_\Sigma$ are equal to total space, base space, and projection of $E^L_\Sigma$ respectively. Let $\mu_{-\Sigma}$ denote the fiber product operation on $E^L_{-\Sigma}$ and let $\mu_\Sigma^{op}$ denote the fiber product operation of $E^L_\Sigma$ with the order of the factors reversed. We prove that fiber product operations $\mu_{-\Sigma}$ and $\mu_\Sigma^{op}$ are equal. The disjoint union of two copies of the cylinder $-\Sigma\times [0,1]$ on $-\Sigma$ and the disjoint union of two copies of the cylinder $\Sigma\times [0,1]$ on $\Sigma$ are equal as manifolds but posses opposite orientation. The manifold $(-\Sigma\times [0,1])\cup_{-\Sigma}(-\Sigma\times [0,1])$ obtained by gluing these two copies the corresponding two copies of $-\Sigma\times [0,1]$ is thus equal to the manifold obtained from gluing operation on the gluing triple formed by the two copies of $\Sigma\times [0,1]$ and $\Sigma$ with gluing ordered reversed. Let $\sigma$ denote the symmetry operation on $L^2_{\Sigma\times [0,1]}$. From the condition of functor $L$ being symmetric monoidal it follows that every diffeomorphism $\Phi$ from $(-\Sigma\times [0,1])\cup_{-\Sigma}(-\Sigma\times [0,1])$ induces a diffeomorphism $L_\Phi$ from $\sigma L_{\Sigma\times [0,1]}^2$ to $L_{\Sigma\times [0,1]}$. Equality of $\mu_{-\Sigma}$ and $\mu_\Sigma^{op}$ follows from this by choosing diffemorphims $\Phi$ in such a way that implements the fiber product operation on $E^L_{-\Sigma}$. This concludes the proof.
\end{proof}

\noindent We call $E^L_0$ the functor of cylinder fibered topological semi-groups associated to $L$. Functors of cylinder topological semi-groups will be object functors of cornered topological quantum field theories associated to equivariant local field theories.

\

\noindent \textsl{Functor of morphisms}

\

\noindent We now define functors of morphisms of cornered topological quantum field theories associated to equivariant local field theories. We now associate, to every equivariant local field theory a symmetric monoidal functor from the category of morphisms of the symmetric monoidal double category of cornered cobordisms of the corresponding dimension to the category of rigid fibered bimodules over rigid fibered topological semi-groups. We being by defining this functor on objects.

\

\noindent \noindent Let $n$ be a positive integer. Let $L$ be an equivariant local field theory of dimension $n$. We will now associate to every oriented cobordism $M$ of dimension $n$ from an oriented manifold $\Sigma$ to an oriented manifold $\Sigma'$ a left-right fibered bimodule $\Omega^L_M$ over cylinder fibered topological semi-groups $E^L_\Sigma$ and $E^L_{\Sigma'}$. We begin with a special case of this. Let $M$ be an oriented manifold of dimension $n$. We associate to $M$ a left fibered module $\Omega^L_M$ over the cylinder fibered topological semi-group $E^L_{\partial M}$ associated to the boundary $\partial M$ of $M$ as follows:  We make the underlying space of $\Omega^L_M$ to be the space underlying the space of global solutions $L_M$ associated to $M$. We will make source function $s$ on $\Omega^L_M$ to be the restriction transformation $r_M$ associated to $M$. To define the left action $\lambda$ of $E^L_{\partial M}$ on $\Omega^L_M$ we first introduce the concept of collar diffeomorphism. Given a codimension 0 properly embedded submanifold $\Sigma$ of the boundary $\partial M$ of an oriented manifold $M$ and a bicollar neighborhood $U$ of $\Sigma$ in $M\cup_\Sigma (\Sigma\times [0,1])$, containing $\Sigma\times [0,1]$ and extending the product structure of $\Sigma\times [0,1]$ we will say that an oriented diffeomorphism

\[\Phi: M\cup_\Sigma (\Sigma\times [0,1])\to M\]

\noindent is a collar diffeomorphism with respect to $U$ if $\Phi$ restricts to the identity function outside $U$ and if $\Phi$ preserves fibers in $U$ and covers the identity diffeomorphism from the top to the bottom copies of $\Sigma$ in $\Sigma\times [0,1]$. Fix now now a collar diffeomorphism $\Phi$ from $(\partial M\times [0,1])\cup_{\partial M}M$ to $M$ with respect to a given collar neighborhood of $\partial M$ on $(\partial M\times [0,1])\cup_{\partial M}M$. In that case we make $\lambda$ to be defined by the following commutative diagram

\begin{center}

\begin{tikzpicture}
  \matrix (m) [matrix of math nodes,row sep=4em,column sep=7em,minimum width=2em]
  {
     E^L_{\partial M}\times_{\pi,s}\Omega^L_M&\Omega^L_M\ \\
     L_{(\partial M\times [0,1])\cup_{\partial M}M}& \\};
  \path[-stealth]
    (m-1-1) edge node [above] {$\lambda$} (m-1-2)
            edge node [left] {$\bullet_{\partial M}$} (m-2-1)
    (m-2-1) edge node [right] {$L_\Phi$} (m-1-2);
\end{tikzpicture}
\end{center}

\noindent where $\pi$ in the fibered product $E^L_{\partial M}\times _{\pi,s}\Omega^L_M$ denotes the projection of fibered semi-group $E^L_{\partial M}$ and where $\bullet_{\partial M}$ denotes the gluing transformation associated to the gluing triple formed by the disjoint union of $M$ and $\partial M\times [0,1]$ and the bottom and top copies of $\partial M$ in $\partial M\times [0,1]$. We prove the following lemma.

\begin{lem}
In the situation above left action $\lambda$ of cylinder fibered topological semi-group $E^L_{\partial M}$ on $\Omega^L_M$ is well defined and provides $E^L_M$ with the structure of a rigid left fibered module over $E^L_{\partial M}$.
\end{lem}

\begin{proof}

Let $n$ be a positive integer. Let $L$ be an equivariant local field theory of dimension $n$. Let $M$ be an oriented manifold of dimension $n$. We wish to prove that left action $\lambda$ of $E^L_{\partial M}$ on $\Omega^L_M$ defined in the previous paragraph does not depend on the specific choice of collar diffeomorphism $\Phi$, that thus defined $\lambda$ provides $\Omega^L_M$ with the structure of a left fibered bimodule over $E^L_{\partial M}$, and that with this structure $\Omega^L_M$ is rigid.

Let $\Psi$ be another collar diffeomorphism from $(\partial M\times [0,1])\cup_{\partial M}M$ to $M$ with respect to a collar neighborhood of $\partial M$ on $(\partial M\times [0,1])\cup_{\partial M}M$. The condition that both $\Phi$ and $\psi$ are collar diffeomorphisms on $(\partial M\times [0,1])\cup_{\partial M}M$ implies that $\Phi$ and $\Psi$ are isotopoically equivalent relative to boundary [wm] and thus by the axiom of re-parametrization invariance their images $L_\Phi$ and $L_\Psi$ under the functor of spaces of global solutions of $L$ are equal. Commutative diagrams defining left action $\lambda$ relative to $\Phi$ and $\Psi$ are thus equal. We conclude that the definition of $\lambda$ is independent of the specific choice of collar diffeomorphism $\Phi$.

We now prove that thus defined left action $\lambda$ provides $\Omega^L_M$ with the structure of a left fibered module over $E^L_{\partial M}$. We begin by proving the compatibility of left action $\lambda$ of $E^L_{\partial M}$ of $\Omega^L_M$ and source function $s$ on $\Omega^L_M$. Consider the following diagram

\begin{center}

\begin{tikzpicture}
  \matrix (m) [matrix of math nodes,row sep=4em,column sep=4em,minimum width=2em]
  {
     L_{\partial M\times [0,1]}\times_{L_{\partial M}}L_M&L_{(\partial M\times [0,1])\cup_{\partial M}M}&L_M \\
     L_{\partial M}^3& L_{\partial M}&L_{\partial M}\\};
  \path[-stealth]
    (m-1-1) edge node [above] {$\bullet^{-1}_{\partial M}$} (m-1-2)
            edge node [left] {$r_{\partial M\times [0,1]}\times r_M$} (m-2-1)
    (m-2-1) edge node [below] {$p$} (m-2-2)
		(m-1-2) edge node [left] {$r_{(\partial M\times [0,1])\cup_{\partial M}M}$} (m-2-2)
		(m-1-2) edge node [above] {$L_\Phi$}(m-1-3)
		(m-1-3) edge node [right] {$r_M$}(m-2-3)
		(m-2-2)edge node [below] {$id_{L_{\partial M}}$}(m-2-3);
\end{tikzpicture}
\end{center}


\noindent where $p$ denotes the projection of $L^3_{\partial M}$ onto the copy of $L_{\partial M}$ corresponding to the factor appearing in the first entry of the product. Arguments analogous to the ones used in the third paragraph of the proof of lemma 1.1 prove that the diagram above commutes. Compatibility of left action $\lambda$ of $E^L_{\partial M}$ on $\Omega^L_M$ and source function $s$ on $\Omega^L_M$ follows from this and from extended monoidality of functor of spaces of germs of local solutions of $L$. We now prove associativity of left action $\lambda$ of $E^L_{\partial M}$ on $\Omega^L_M$.

Let $\Phi$ be a collar diffeomorphism from $(\partial M\times [0,1])\cup_{\partial M}M$ to $M$ relative to a collar neighborhood of $\partial M$ on $(\partial M\times [0,1])\cup_{\partial M}M$. Let $\Psi$ be a collar diffeomorphism from $(\partial M\times [0,1])\cup_{\partial M}(\partial M\times [0,1])$ to $\partial M\times [0,1]$ relative to a collar neighborhood of $\partial M$ on $(\partial M\times [0,1])\cup_{\partial M}(\partial M\times [0,1])$. Compositions $\Phi(id_{\partial M\times [0,1]}\cup_{id_{\partial M}}\Phi)$ and $\Phi(\Psi\cup_{\partial M}id_M)$ both clearly preserve fibers and cover the identity diffeomorphism of $\partial M$. From this it follows that the following diagram is commutative 

\begin{center}

\begin{tikzpicture}
  \matrix (m) [matrix of math nodes,row sep=4em,column sep=7em,minimum width=2em]
  {
     L_{(\partial M\times [0,1])\cup_{\partial M}(\partial M\times [0,1])\cup_{\partial M}M}&L_{(\partial M\times [0,1])\cup_{\partial M}M}\ \\
     L_{(\partial M\times [0,1])\cup_{\partial M}M}& L_M \\};
  \path[-stealth]
    (m-1-1) edge node [above] {$L_{(id_{\partial M\times [0,1]})\cup_{\partial M}\Phi}$} (m-1-2)
            edge node [left] {$L_{(\Psi\cup_{\partial M}id_M)}$} (m-2-1)
    (m-2-1) edge node [below] {$L_\Phi$} (m-2-2)
		(m-1-2) edge node [right] {$L_\Phi$} (m-2-2);
\end{tikzpicture}
\end{center}

\noindent Associativity of left action $\lambda$ of $E^L_{\partial M}$ on $\Omega^L_M$ follows directly from this together with the fact that restriction transformation $r$ is natural and monoidal. This proves that left action $\lambda$ of $E^L_{\partial M}$ on $\Omega^L_M$ defines the structure of a left fibered module $\Omega^L_M$ over $E^L_{\partial M}$. The fact that left action $\lambda$ of $E^L_{\partial M}$ on $\Omega^L_M$ is a homeomorphism follows directly from the gluing transformation axiom in the definition of $L$. We conclude that with the structure described $\Omega^L_M$ is a rigid left fibered module over cylinder topological semi-group $E^L_{\partial M}$ associated to the boundary $\partial M$ of $M$. This concludes the proof.
\end{proof}

\noindent Let now $\Sigma$ and $\Sigma'$ be oriented manifolds of dimension $n-1$ and let $M$ be an oriented cobordism from $\Sigma$ to $\Sigma'$. In that case from the above lemma and from lemma ? it follows that $\Omega^L_M$ admits the structure of a left module over fibered topological semi-group $E^L_\Sigma\times E^{L^{op}}_{\Sigma'}$ where $E^{L^{op}}_{\Sigma'}$ denotes fibered topological semi-group $E^L_{\Sigma'}$ with the order on the product operation reversed. The component on $E^{L^{op}}_{\Sigma'}$ of this action induces the structure of a right fibered module over $E^L_{\Sigma'}$ on $\Omega^L_M$. The obvious fact that the composition of a collar diffeomorphim from $(\Sigma\times [0,1])\cup_\Sigma M$ to $M$ and a collar diffeomorphism from $M\cup_\Sigma' (\Sigma'\times [0,1])$ to $M$ defined by disjoint collar neighborhoods of $\Sigma$ and $\Sigma'$ is isotopically equivalent relative to boundary to a collar diffeomorphism from $(\partial M\times [0,1])\cup_{\partial M}M$ to $M$ implies that the actions on $\Omega^L_M$ corresponding the components on $E^L_\Sigma$ and $E^L_{\Sigma'}$ commute. We conclude that in this case $\Omega^L_M$ admits the structure of a left-right fibered bimodule over $E^L_\Sigma$ and $E^L_{\Sigma'}$. We call $\Omega^L_M$ the cylinder fibered bimodule associated to cobordism $M$ by $L$. We now extend the the cylinder fibered bimodule construction to equivariant diffeomorphisms. We first prove the following lemma.

\begin{lem}
Let $n$ be a positive integer. Let $L$ be an equivariant local field theory of dimension $n$. Let $M$ and $N$ be oriented manifolds of dimension $n$. Let $\Phi$ be a diffeomorphism from $M$ to $N$. Denote by $\varphi$ the restriction of $\Phi$ to the boundary $\partial M$ of $M$. In that case the pair $(L_\varphi,L_\Phi)$ is an equivariant isomorphism of fibered left modules from $\Omega^L_M$ to $\Omega^L_N$.
\end{lem}

\begin{proof}
Let $n$ be a positive integer. Let $L$ be an equivariant local field theory of dimension $n$. Let $M$ and $N$ be oriented manifolds of dimension $n$. Let $\Phi$ be a diffeomorphism from $M$ to $N$. Denote by $\varphi$ the restriction of $\Phi$ to the boundary $\partial M$ of $M$. We wish to prove in that case that the pair $(L_\varphi,L_\Phi)$ is an equivariant isomorphism of fibered left modules from $\Omega^L_M$ to $\Omega^L_N$.

Compatibility of $L_\Phi$ and $L_\varphi$ with source functions of $\Omega^L_M$ and $\Omega^L_N$ follows directly from naturality of restriction transformation $r$ of $L$. We prove that $L_\Phi$ and $L_\varphi$ are compatible with left actions of $E^L_{\partial M}$ and $E^L_{\partial N}$ on $\Omega^L_M$ and $\Omega^L_N$ respectively. Let $\lambda$ denote the left action of $E^L_{\partial M}$ on $\Omega^L_M$ and let $\lambda'$ denote the left action of $E^L_{\partial N}$ on $\Omega^L_N$. Let $\Psi$ be a collar diffeomorphism from $(\partial M\times [0,1])\cup_{\partial M}M$ to $M$ relative to a collar neighborhood of $\partial M$ in $(\partial M\times [0,1])\cup_{\partial M}M$ and let $\Psi'$ be a collar diffeomorphism from $(\partial N\times [0,1])\cup_{\partial N}N$ to $N$ relative to a collar neighborhood of $\partial N$ in $(\partial N\times [0,1])\cup_{\partial N}N$. In this case compositions $\Phi\Psi$ and $\Psi'(\Phi\cup_\varphi(\varphi\times id_{[0,1]}))$ both preserve fibers and dominate a common diffeomorphism $\varphi$ and are thus isotopically equivalent relative to boundary. The following diagram is thus commutative

\begin{center}

\begin{tikzpicture}
  \matrix (m) [matrix of math nodes,row sep=4em,column sep=7em,minimum width=2em]
  {
     L_{(\partial M\times [0,1])\cup_{\partial M}}&L_{(\partial N\times [0,1])\cup_{\partial N}N}\ \\
     L_M& L_N \\};
  \path[-stealth]
    (m-1-1) edge node [above] {$L_{\Phi\cup_\varphi(\varphi\times id_{[0,1]})}$} (m-1-2)
            edge node [left] {$L_\Psi$} (m-2-1)
    (m-2-1) edge node [below] {$L_\Phi$} (m-2-2)
		(m-1-2) edge node [right] {$L_{\Psi'}$} (m-2-2);
\end{tikzpicture}
\end{center}

\noindent Compatibility of $L_\Phi$ and $L_\varphi$ with $\lambda$ and $\lambda'$ follows directly from this and from naturality of the gluing transformation $\bullet$ of $L$. This concludes the proof. 
\end{proof}

\begin{cor}
Let $n$ be a positive integer. Let $L$ be an equivariant local field theory of dimension $n$. Let $\Sigma,\Sigma',\Lambda$ and $\Lambda'$ be oriented manifolds of dimension $n-1$. Let $M$ be an oriented cobordism from $\Sigma$ to $\Lambda$ and let $M'$ be an oriented cobordism from $\Sigma'$ to $\Lambda'$. Let $\Phi$ be an oriented diffeomorphism from $M$ to $M'$ restricting to a diffeomorphism from $\Sigma$ to $\Sigma'$ and to a diffeomorphism from $\Lambda$ to $\Lambda'$. Write $\varphi$ and $\psi$ for these restrictions. In that case the triple $(E^L_\varphi,L_\Phi,E^L_\psi)$ is an equivariant isomorphism from $\Omega^L_M$ to $\Omega^L_{M'}$.
\end{cor}

\noindent In the situation of the lemma above we write $\Omega^L_\Phi$ for the isomorphism of fibered bimodules $(E^L_\varphi,L_\Phi, E^L_\psi)$. We call isomorphism $E^L_\Phi$ the cylinder equivariant isomorphism associated to diffeomorphism $\Phi$. We denote by $E^L_1$ the symmetric monoidal functor from the category $\tilde{\mbox{\textbf{Cob}}}(n)_1$ of cornered oriented cobordisms of dimension $n$ to the category $\tilde{\mbox{\textbf{fsGrp}}}_1$ of rigid fibered bimodules over rigid fibered topological semi-groups associating cylinder fibered bimodule $\Omega^L_M$ to every oriented cobordism $M$ of dimension $n$ and associating cylinder equivariant isomorphism $\Omega^L_\Phi$ to every equivariant diffeomorphism $\Phi$ of cornered cobordisms. We write $E^L_1$ for the functor of cylinder bimodules associated to equivariant local field theory $L$. Thus defined $E^L_1$ is easily seen to be symmetric monoidal. Functors of cylinder fibered bimodules will serve as morphism functors for cornered topological quantum field theories associated to equivariant local field theories.

\

\noindent \textsl{Main theorem}

\

\noindent We now present a proof of the main result of this section.

\begin{thm}
Let $n$ be a positive integer. Let $L$ be an equivariant local field theory of dimension $n$. In that case the pair $E^L_0,E^L_1$ formed by the cylinder fibered topological semi-group functor and by the cylinder fibered bimodule functor  forms a symmetric monoidal double functor from double category $\tilde{\mbox{\textbf{Cob}}}(n)$ of oriented cornered cbordisms of dimension $n$ to double category $\tilde{\mbox{\textbf{fsGrp}}}$ of rigid fibered topological semi-groups.
\end{thm} 

\begin{proof}
Let $n$ be a positive integer. Let $L$ be an equivariant local field theory of dimension $n$. We wish to prove in this case that pair formed by functor $E^L_0$ of cylinder fibered topological semigroups and functor $E^L_1$ of cylinder fibered bimodules forms a symmetric monoidal double functor from double cateogry $\tilde{\mbox{\textbf{Cob}}}(n)$ of cornered cobordisms of dimension $n$ to double category $\tilde{\mbox{\textbf{fsGrp}}}$ of rigid fibered topological semi-groups.

Functors $E^L_0$ and $E^L_1$ are symmetric monoidal and are clearly compatible with source and target functors on $\tilde{\mbox{\textbf{Cob}}}(n)$ and $\tilde{\mbox{\textbf{fsGrp}}}$. We prove that $E^L_0$ and $E^L_1$ are compatible with horizontal identity transformations on $\tilde{\mbox{\textbf{Cob}}}(n)$ and $\tilde{\mbox{\textbf{fsGrp}}}$ that is, we wish to prove that the following diagram of functors commutes

\begin{center}

\begin{tikzpicture}
  \matrix (m) [matrix of math nodes,row sep=4em,column sep=7em,minimum width=2em]
  {
     \tilde{\mbox{\textbf{Cob}}}(n)_0&\tilde{\mbox{\textbf{fsGrp}}}_0\ \\
     \tilde{\mbox{\textbf{Cob}}}(n)_1&\tilde{\mbox{\textbf{fsGrp}}}_1 \\};
  \path[-stealth]
    (m-1-1) edge node [above] {$E_0^L$} (m-1-2)
            edge node [left] {$\times id_{[0,1]}$} (m-2-1)
    (m-2-1) edge node [below] {$E_1^L$} (m-2-2)
		(m-1-2) edge node [right] {$i$} (m-2-2);
\end{tikzpicture}
\end{center}

\noindent We first prove the commutativity of the above diagram on objects. Let $\Sigma$ be an oriented manifold of dimension $n-1$. We first prove that the horizontal identity $i_{E^L_\Sigma}$ of the cylinder fibered topological semi-group $E^L_\Sigma$ associated to $\Sigma$ and cylinder fibered bimodule $\Omega^L_{\Sigma\times [0,1]}$ associated to the cylinder $\Sigma\times [0,1]$ on $\Sigma$ are equal as fibered left-right bimodules over $E^L_\Sigma$. The underlying space of $i_{E^L_\Sigma}$ equals the total space of fibered topological semi-group $E^L_\Sigma$, which in turn equals the space of global solutions $L_{\Sigma\times [0,1]}$ on the cylinder $\Sigma\times [0,1]$. This is equal to the underlying space of cylinder fibered bimodule $\Omega^L_{\Sigma\times [0,1]}$ associated to $\Sigma\times [0,1]$. Underlying spaces of $i_{E^L_\Sigma}$ and $\Omega^L_{\Sigma\times [0,1]}$ are thus equal. Source and target functions on $i_{E^L_\Sigma}$ are both equal to projection $\pi$ on $E^L_\Sigma$. Now, projection $\pi$ on $E^L_\Sigma$ equals composition $pp_{L^2_\Sigma}r_{\Sigma\times [0,1]}$ where $p_{L_\Sigma^2}$ denotes the projection of $L_{\partial (\Sigma\times [0,1])}$ onto the factor $L^2_\Sigma$ corresponding to the top and bottom copies of $\Sigma$ on $\partial (\Sigma\times [0,1])$ and where $p$ denotes any of the two projections of $L_\Sigma^2$ onto $L_\Sigma$. This is equal to composition of restriction transformation $r_{\Sigma\times [0,1]}$ and projection of $L_{\partial(\Sigma\times [0,1])}$ onto the factor $L_\Sigma$ corresponding to any of the the top or bottom copies of $\Sigma$ on $\Sigma\times [0,1]$. This equals source and target functions on $\Omega^L_{\Sigma\times [0,1]}$. Source and target functions on $i_{E^L_\Sigma}$ and $\Omega^L_{\Sigma\times [0,1]}$ are thus equal. We now prove that the left actions of $E^L_\Sigma$ on $i_{E^L_\Sigma}$ and $\Omega^L_{\Sigma\times [0,1]}$ are equal. Left action $\lambda$ on $i_{E^L_\Sigma}$ is defined by diagram

\begin{center}

\begin{tikzpicture}
  \matrix (m) [matrix of math nodes,row sep=4em,column sep=7em,minimum width=2em]
  {
     L_{\Sigma\times [0,1]}\times_{L_\Sigma} L_{\Sigma\times [0,1]}&L_{\Sigma\times [0,1]} \\
     L_{(\Sigma\times [0,1])\cup_\Sigma(\Sigma\times [0,1])}& \\};
  \path[-stealth]
    (m-1-1) edge node [above] {$\lambda$} (m-1-2)
            edge node [left] {$\bullet^{-1}_{\partial M}$} (m-2-1)
    (m-2-1) edge node [right] {$L_\Phi$} (m-1-2);
\end{tikzpicture}
\end{center}

\noindent for any diffeomorphism $\Phi$ from $(\Sigma\times [0,1])\cup_\Sigma(\Sigma\times [0,1])$ to $\Sigma\times [0,1]$ preserving fibers and covering the identity diffeomorphism on the top and bottom copies of $\Sigma$ on both cylinders. Choosing $\Phi$ to be a collar diffeomorphism from $(\Sigma\times [0,1])\cup_\Sigma(\Sigma\times [0,1])$ to $\Sigma\times [0,1]$ we conclude that left action of $E^L_\Sigma$ on $i_{E^L_\Sigma}$ equals left action of $E^L_\Sigma$ on $\Omega^L_\Sigma$. Analogously the right action of $E^L_\Sigma$ on $i_{E^L_\Sigma}$ equals the right action of $E^L_\Sigma$ on $\Omega^L_{\Sigma\times [0,1]}$. We conclude that left-right fibered bimodules $i_{E^L_\Sigma}$ and $\Omega^L_{\Sigma\times [0,1]}$ over $E^L_\Sigma$ are equal. The compatibility diagram for horizontal identities above thus commutes on objects. We now prove it commutes on morphisms. Let $\Sigma$ and $\Sigma'$ oriented manifolds of dimension $n-1$. Let $\varphi$ be an oriented diffeomorphism from $\Sigma$ to $\Sigma'$. We wish to prove that equivariant diffeomorphisms of fibered bimodules $i_{E^L_\varphi}$ and $\Omega^L_{\varphi\times id_{[0,1]}}$ are equal. The fact that the product of the restriction of $\varphi\times id_{[0,1]}$ to the bottom or the top copies of $\Sigma$ on the boundary $\partial(\Sigma\times [0,1])$ of $\Sigma\times [0,1]$ and the identity $id_{[0,1]}$ of the interval $[0,1]$ again equals $\varphi\times id_{[0,1]}$ implies that equivariant isomorphism $\Omega^L_{\varphi\times id_{[0,1]}}$ equals triple $(L_{\varphi\times id_{[0,1]}},L_{\varphi\times id_{[0,1]}},L_{\varphi\times id_{[0,1]}})$. This in turn is equal to $i_{E^L_\varphi}$. Equivariant isomorphisms $i_{E^L_\varphi}$ and $\Omega^L_{\varphi\times id_{[0,1]}}$ are thus equal. We conclude that the compatibility square for horizontal identities commutes on morphisms, that is, we conclude that functors $E^L_0$ and $E^L_1$ are compatible with horizontal identity functors of double categories $\tilde{\mbox{\textbf{Cob}}}(n)$ and $\tilde{\mbox{\textbf{fsGrp}}}$. 

We now prove weak compatibility of functors $E^L_0$ and $E^L_1$ with horizontal composition operations in double categories $\tilde{\mbox{\textbf{Cob}}}(n)$ and $\tilde{\mbox{\textbf{fsGrp}}}$, that is we now provide a coherent monoidal natural isomorphism $A$ from the top-right corner to the bottom-left corner of the following diagram

\begin{center}

\begin{tikzpicture}
  \matrix (m) [matrix of math nodes,row sep=4em,column sep=5em,minimum width=2em]
  {
     \tilde{\mbox{\textbf{Cob}}}(n)_1\times_{\tilde{\mbox{\textbf{Cob}}}(n)_0}\tilde{\mbox{\textbf{Cob}}}(n)_1 &\tilde{\mbox{\textbf{fsGrp}}}_1\times_{\tilde{\mbox{\textbf{fsGrp}}}_0}\tilde{\mbox{\textbf{fsGrp}}}_1\ \\
     \tilde{\mbox{\textbf{Cob}}}(n)_1&\tilde{\mbox{\textbf{fsGrp}}}_1 \\};
  \path[-stealth]
    (m-1-1) edge node [above] {$E_1^L\times_{E^L_0} E^L_1$} (m-1-2)
            edge node [left] {$\ast$} (m-2-1)
    (m-2-1) edge node [below] {$E_1^L$} (m-2-2)
		(m-1-2) edge node [right] {$\circledast$} (m-2-2);
\end{tikzpicture}
\end{center}

\noindent Let $\Sigma,\Sigma'$ and $\Sigma''$ be oriented manifolds of dimension $n-1$. Let $M$ be an oriented cobordism from $\Sigma$ to $\Sigma'$ and let $N$ be an oriented cobordism from $\Sigma'$ to $\Sigma''$. The space underlying left-right fibered bimodule $\Omega^L_{M\cup_{\Sigma'}N}$ over $E^L_\Sigma$ and $E^L_{\Sigma''}$ associated to glued cobordism $M\cup_{\Sigma'}N$ from $\Sigma$ to $\Sigma''$ is equal to the space of global solutions $L_{\Sigma\times [0,1]}$ of the cylinder $\Sigma\times [0,1]$ on $\Sigma$. The space underlying horizontal composition $\Omega^L_M\circledast_{E^L_{\Sigma'}}\Omega^L_N$ of cylinder fibered bimodule $\Omega^L_M$ associated to $M$ and cylinder fibered bimodule $\Omega^L_N$ associated to $N$ is equal to product $L_M\times_{L_{\Sigma'}}L_N$ of spaces of global solutions $L_M$ and $L_N$ associated to $M$ and $N$ fibered by space of germs of local solutions $L_\Sigma$ associated to $\Sigma$. Gluing transformation $\bullet_{\Sigma'}$ associated to the gluing pair formed by the disjoint union of $M$ and $N$ and common boundary piece $\Sigma'$ defines a homeomorphism from the space underlying $\Omega^L_{M\cup_\Sigma N}$ to the space underlying $\Omega^L_M\circledast_{E^L_{\Sigma'}}\Omega^L_N$. We write $A_{M,N}$ for triple $(id_{E^L_\Sigma},\bullet_{\Sigma'},id_{E^L_{\Sigma''}})$. We prove that thus defined $A_{M,N}$ is an equivariant morphism of fibered bimodules. Compatibility of $A_{M,M'}$ with source and target functions of $\Omega^L_{M\cup_\Sigma N}$ and $\Omega^L_M\circledast_{E^L_{\Sigma'}}\Omega^L_N$ respectively follows directly from symmetric monoidality of functor of spaces of global slutions of $L$, monoidality of restriction transformation $r$ of $L$, and from compatibility of $r$ and gluing transformation $\bullet$ of $L$. We prove compatibility of $A_{M,N}$ with left actions of $E^L_\Sigma$ on $\Omega^L_{M\cup_{\Sigma'}}\Omega^L_N$ and $\Omega^L_M\circledast_{E^L_{\Sigma'}}\Omega^L_N$ respectively. Consider the following diagram

\begin{center}

\begin{tikzpicture}
  \matrix (m) [matrix of math nodes,row sep=3em,column sep=8em,minimum width=2em]
  {
     L_{\Sigma\times [0,1]}\times_{L_\Sigma}L_{M\cup_{\Sigma'}N}&L_{\Sigma\times [0,1]}\times_{L_{\Sigma}}L_M\times_{L_{\Sigma'}}L_N \\
     L_{(\Sigma\times [0,1])\cup_\Sigma M\cup_{\Sigma'}N}&L_{(\Sigma\times [0,1])\cup_{\Sigma'}M}\times_{L_{\Sigma'}}L_N \\
		L_{M\cup_{\Sigma'}N}&L_M\times_{L_{\Sigma'}}L_N\\};
  \path[-stealth]
    (m-1-1) edge node [above] {$L_{id_{\Sigma\times [0,1]}}\times \bullet_{\Sigma'}$} (m-1-2)
            edge node [left] {$\bullet_\Sigma^{-1}$} (m-2-1)
    (m-1-2) edge node [right] {$\bullet_{\Sigma}\times id_N$} (m-2-2)
		(m-2-1) edge node [above] {$\bullet_{\Sigma'}$} (m-2-2)
		(m-2-1) edge node [left] {$L_{\Phi\cup id_N}$} (m-3-1)
		(m-2-2) edge node [right] {$L_{\Phi}\times id_N$} (m-3-2)
		(m-3-1) edge node [below] {$\bullet_{\Sigma'}$} (m-3-2);
\end{tikzpicture}
\end{center}

\noindent where $\Phi$ is any collar diffeomorphism from $(\Sigma\times [0,1])\cup_\Sigma M$ to $M$ relative to a collar neighborhood of $\Sigma$ on $(\Sigma\times [0,1])\cup_\Sigma M$ disjoint from $\Sigma'$ in $M$. Commutativity of the upper square in the above diagram follows directly from associativity of gluing transformation $\bullet$ of $L$. Commutativity of the lower square follows from naturality of $\bullet$. The diagram above is thus commutative. Compatibility of $A_{M,N}$ with left action of $E^L_\Sigma$ on $\Omega_{M\cup_{\Sigma'}N}$ and left action of $E^L_\Sigma$ on $\Omega^L_M\circledast_{E^L_{\Sigma'}}\Omega^L_N$ follows from this. A similar argument proves compatibility of $A_{M,N}$ with right actions of $E^L_{\Sigma''}$ on $\Omega^L_{M\cup_{\Sigma'}}N$ and $\Omega^L_M\circledast_{E^L_{\Sigma'}}\Omega^L_N$. We conclude that $A_{M,N}$ is an equivariant isomorphism of fibered bimodules from $\Omega^L_{M\cup_{\Sigma'}N}$ to $\Omega^L_M\circledast_{E^L_{\Sigma'}}\Omega^L_N$. We write $A$ for the collection of all equivariant isomorphisms of the form $A_{M,N}$ with $M,N$ running through all horizontally compatible pairs of oriented cobordisms of dimension $n$. Thus defined $A$ is a natural transformation from the upper-right corner of the compatibility diagram for horizontal composition to its lower-left corner. Naturality and monoidality of $A$ follows directly from naturality and monoidality of gluing transformation $\bullet$ of $L$. It remains to prove coherence of $A$ but the hexagon axiom for $A$ follows directly from associativity of $\bullet$. This concludes the proof-

\end{proof}

\noindent We will write $E^L$ for the pair formed by symmetric monoidal functors $E^L_0$ and $E^L_1$. Theorem 4.7 says that thus defined $E^L$ is a cornered topological quantum field theory of the same dimension as equivariant local field theory $L$. We call $E^L$ the cylinder cornered topological quantum field theory associated to $L$. We have thus associated, through the cylinder cornered topological quantum field theory construction, a cornered topological quantum field theory to every equivariant local field theory.

\section{Smoothing the cylinder construction}

\

\noindent In this section we present a smooth version of cylinder topological quantum field theories associated to equivariant local field theories. We introduce the notions of fibered Lie semi-group and of smooth fibered bimodule. We prove that there is no natural notion of a double category of fibered Lie semi-groups ans smooth fibered bimodules. We associate nevertheless to every equivariant local field theory a topological quantum field theory of fibered Lie semi-groups and smooth fibered bimodules, concrete over the cylinder topological quantum field theory associated to $L$.We begin with the definition of fibered Lie semi-group.

\

\noindent \textsl{Smoothing cylinder fibered semi-groups}

\

\noindent Let $E$ be a fibered topological semi-group with base space $X$, projection $\pi$ and fiber product operation $\mu$. We will understand for a structure of fibered Lie semi-group on $E$ a structure of smooth manifold on the total space of $E$ and a structure of smooth manifold on the base space $X$ of $E$ such that with these structures projection $\pi$ of $E$ is a smooth function, topological fiber product $E\times_\pi E$ of $E$ with itself relative to $\pi$ admits the structure of a smooth manifold in such a way that with this structure $E\times _\pi E$ is the fiber product, in the category of smooth manifolds and smooth functions, of $E$ with itself relative to $\pi$, and finally such that with this structure fiber product operation $\mu$ on $E$ is a smooth function from $E\times_\pi E$ to $E$. We will say that a fibered topological semi-group together with a structure of fibered Lie semi-group is a fibered Lie semi-group. In analogy with the topological case we will say that a fibered Lie semi-group $E$ with smooth projection $\pi$ and smooth fiber product operation $\mu$ is rigid if $\mu$ is a diffeomorphism from $E\times_\pi E$ to $E$. 

We will say that a morphism of topological fibered semi-groups $(\varphi,\Phi)$ from the underlying topological fibered semi-group of a fibered Lie semi-group $E$ to the underlying topological fibered semi-group of another fibered Lie semi-group $E'$ is a morphism of fibered Lie semi-groups from $E$ to $E'$ if both base space function $\varphi$ and total space function $\Phi$ of $(\varphi,\Phi)$ are smooth. We will write \textbf{fsGrp}$^\ell_0$ for the category of fibered Lie semi-groups and their morphisms. Cartesian product provides \textbf{fsGrp}$^\ell_0$ with the structure of a symmetric monoidal category. We will write $\tilde{\mbox{\textbf{fsGrp}}}^\ell_0$ for the full symmetric monoidal subcategory of \textbf{fsGrp}$^\ell_0$ generated by rigid fibered Lie semi-groups. We will denote by $U_0$ the symmetric monoidal functor from category \textbf{fsGrp}$^\ell_0$ to category \textbf{fsGrp}$_0$ of topological fibered semi-groups associating to every fibered Lie semi-group $E$ its underlying topological fibered semi-group and associating to every morphism of fibered Lie semi-groups $(\varphi,\Phi)$ morphism $(\varphi,\Phi)$ itself considered as a morphism between topological fibered semi-groups. Thus defined $U_0$ provides category \textbf{fsGrp}$^\ell_0$ with the structure of a symmetric monoidal category concrete over \textbf{fsGrp}$_0$. We will keep writing $U_0$ for the restriction of $U_0$ to category $\tilde{\mbox{\textbf{fsGrp}}}^\ell_0$. Thus defined $U_0$ now defines the structure of a symmetric monoidal category on $\tilde{\mbox{\textbf{fsGrp}}}^\ell_0$ concrete over category $\tilde{\mbox{\textbf{fsGrp}}}_0$ of rigid topological fibered semi-groups.

\

\noindent Let $n$ be a positive integer. Let $L$ be an equivariant local field theory of dimension $n$. Given an oriented manifold $\Sigma$ of dimension $n-1$ field theory $L$ associates to both the base space $L_\Sigma$ and the total space $L_{\Sigma\times [0,1]}$ of the cylinder topological fibered semi-group $E^L_\Sigma$ associated to $\Sigma$ structures of smooth manifolds. With this structure projection $\pi$ of $E^L_\Sigma$ is a smooth function, by the gluing axiom in the definition of $L$ the topological fiber product $E^L_\Sigma\times_\pi E^L_\Sigma$ admits the structure of a smooth manifold making $E^L_\Sigma\times_\pi E^L_\Sigma$ the fiber profuct of $E^L_\Sigma$ with itself with respect to $\pi$ in the category of smooth manifolds and smooth functions, and making the gluing transformation and thus the fiber product operation on $E^L_\Sigma$ smooth. Equivariant local field theory $L$ thus associates to every oriented manifold $\Sigma$ of dimension $n-1$ the structure of a Lie fibered semi-group on the cylinder topological semi-group $E^L_\Sigma$ associated to $\Sigma$. Observe that with this structure $E^L_\Sigma$ is a reduced fibered Lie semi-group. We call this reduced fibered Lie semi-group the cylinder fibered Lie semi-group associated to $\Sigma$.

Given an oriented diffeomorphism $\varphi$ from an oriented manifold $\Sigma$ of dimension $n-1$ to another oriented manifold $\Sigma'$ of dimension $n-1$ the image $L_\varphi$ of $\varphi$ under the functor of germs of local solutions of $L$ is a smooth function from the base manifold $L_\Sigma$ of the cylinder fibered Lie semi-group $E^L_\Sigma$ to the base manifold of the cylinder fibered Lie semi-group $E^L_{\Sigma'}$, and the image $L_{\varphi\times id_{[0,1]}}$ of diffeomorphism $\varphi\times id_{[0,1]}$ under the functor of global solutions of $L$ is a smooth function from the total manifold $L_{\Sigma\times [0,1]}$ of $E^L_\Sigma$ to the total manifold $L_{\Sigma'\times [0,1]}$ of $E^L_{\Sigma'}$. Morphism of topological fibered semi-groups $E^L_\varphi$ is thus a morphism of fibered Lie semi-groups from cylinder fibered Lie semi-group $E^L_\Sigma$ to fibered Lie semi-group $E^L_{\Sigma'}$. Write $S^L_0$ for the functor from the category of objects $\tilde{\mbox{\textbf{Cob}}}(n)_0$ of the symmetric monoidal double category of cornered cobordisms of dimension $n$ to the symmetric monoidal category $\tilde{\mbox{\textbf{fsGrp}}}^\ell_0$ of rigid fibered Lie semi-groups associating to every oriented manifold of dimension $n-1$ the cylinder fibered Lie semi-group $E^L_\Sigma$ associated to $\Sigma$ and to every oriented diffeomorphism $\varphi$ between oriented manifolds of dimension $n-1$ the morphism $E^L_\varphi$. Thus defined $S^L_0$ is easily seen to be symmetric monoidal and to satisfy equation

\begin{center}

\begin{tikzpicture}
  \matrix (m) [matrix of math nodes,row sep=3em,column sep=7em,minimum width=2em]
  {
     \tilde{\mbox{\textbf{Cob}}}(n)_0& \tilde{\mbox{\textbf{fsGrp}}}^\ell_0 \\
      & \tilde{\mbox{\textbf{fsGrp}}}_0 \\};
  \path[-stealth]
    (m-1-1) edge node [above] {$S^L_0$} (m-1-2)
    (m-1-2) edge node [right] {$U_0$} (m-2-2)
    (m-1-1) edge node [below] {$E^L_0$}(m-2-2);
\end{tikzpicture}
\end{center}

\noindent We call symmetric monoidal functor $S^L_0$ the cylinder fibered Lie semi-group functor. We regard cylinder fibered semi-group functors as smoothings of cylinder fibered topological semi-group functors presented in the previous section.

\

\noindent \textsl{Smoothing fibered bimodules}

\

\noindent In analogy with the sonstruction presented above we now present a smoothing procedure for the cylinder fibered bimodule construction. We begin with the definition of smooth fibered bimodule.

\

\noindent Let $E$ and $E'$ be fibered Lie semi-groups with projections $\pi$ and $\pi'$. Let $\Omega$ be a left-right topological fibered bimodule over topological fibered semi-groups underlying $E$ and $E'$ with source and targer functions $s$ and $t$ and left and right actions $\lambda$ and $\rho$ respectively. We will understand for a structure of a left-right smooth fibered bimodule over fibered Lie semi-groups $E$ and $E'$ on $\Omega$ a structure of smooth manifold on $\Omega$ such that with this structure source and target functions $s$ and $t$ on $\Omega$ are smooth, such that the topological fiber products $E\times_{\pi,s}\Omega$ and $\Omega\times_{t,\pi'}E'$ both admit smooth structures making them fiber products of the corresponding diagrams in the category of smooth manifolds and smooth functions, and such that, with this structure left and right actions $\lambda$ and $\rho$ are smooth functions from $E\times_{\pi,s}\Omega$ and $\Omega\times_{t,\pi'}E'$ to $\Omega$ respectively. We will say that a fibered bimodule together with a structure of smooth fibered bimodule is a smooth fibered bimodule. Assuming fibered Lie semi-groups $E$ and $E'$ are rigid we will say that left-right smooth fibered bimodule $\Omega$ over $E$ and $E'$ is rigid if both the left action $\lambda$ of $E$ on $\Omega$ and the right action $\rho$ of $E'$ on $\Omega$ are diffeomorphisms from $E\times_{\pi,s}\Omega$ to $\Omega$ and from $\Omega\times_{t,\pi'}E'$ to $\Omega$ respectively. 

We will say that an equivariant morphism $(\varphi,\Phi,\psi)$ from the fibered bimodule underlying a left-right smooth fibered bimodule $\Omega$ over fibered Lie groups $E$ and $F$ to the fibered bimodule underlying another left-right smooth fibered bimodule $\Omega'$ over fibered Lie semi-groups $E'$ and $F'$ is smooth if $\Phi$ is smooth and if $\varphi$ and $\psi$ are morphisms of fibered Lie semi-groups. We will write \textbf{fsGrp}$^\ell_1$ for the category of left-right smooth fibered bimodules over fibered Lie semi-groups and smooth equivariant morphisms. Cartesian product provides \textbf{fsGrp}$^\ell_1$ with the structure of a symmetric monoidal category. We will write $\tilde{\mbox{\textbf{fsGrp}}}^\ell_1$ for the symmetric monoidal sub-cateogry of \textbf{fsGrp}$^\ell_1$ generated by rigid smooth fibered bimodules. We will now write $U_1$ for the symmetric monoidal functor from \textbf{fsGrp}$^\ell_1$ to category \textbf{fsGrp}$_1$ of fibered topological bimodules over fibered topological semi-groups associating to every smooth fibered bimodule its underlying topological fibered bimodule and to every smooth equivariant morphism of smooth fibered bimodules $(\varphi,\Phi,\psi)$, triple $(\varphi,\Phi,\psi)$ itself now considered as a morphism of topological fibered bimodules. We will keep writing $U_1$ for the restriction of $U_1$ to category $\tilde{\mbox{\textbf{fsGrp}}}^\ell_1$. Thus defined $U_1$ now defines the structure of a symmetric monoidal category on $\tilde{\mbox{\textbf{fsGrp}}}^\ell_1$ concrete over category $\tilde{\mbox{\textbf{fsGrp}}}_1$ of rigid topological fibered bimodules.

\

\noindent Let $n$ be a positive integer. Let $L$ be an equivariant local field theory of dimension $n$. Given an oriented cobordism $M$ of dimension $n$ from an oriented manifold $\Sigma$ to another oriented manifold $\Lambda$ field theory associates to the space $L_M$ underlying cylinder fibered bimodule $\Omega^L_M$ associated to $M$, the structure of a smooth manifold.  With this structure source and target functions $s$ and $t$ on $\Omega^L_M$ are smooth. By the gluing axiom in the definition of $L$ the topological fiber products $E^L_\Sigma\times_{\pi,s}\Omega^L_M$ and $\Omega^L_M\times_{t,\pi'}E^L_\Lambda$ where $\pi$ and $\pi'$ denote the projections of cylinder fibered Lie semi-groups $E^L_\Sigma$ and $E^L_\Lambda$, admit structures of smooth manifolds making $E^L_\Sigma\times_{s,\pi}\Omega^L_M$ and $\Omega^L_M\times_{t,\pi'}E^L_\Lambda$ the fiber products of the corresponding diagrams in the category of smooth manifolds and smooth functions, and making the left action $\lambda$ of $E^L_\Sigma$ on $\Omega^L_M$ and the right action of $E^L_\Lambda$ on $\Omega^L_M$ diffeomorphisms. Equivariant local field theory $L$ thus associates to every oriented cobordism of dimension $n$ from an oriented manifold $\Sigma$ to another oriented manifold $\Lambda$ a left-right reduced smooth fibered bimodule $\Omega^L_M$ over cylinder fibered Lie semi-groups $E^L_\Sigma$ and $E^L_\Lambda$. We call cylinder fibered bimodule $\Omega^L_M$ with this structure the smooth cylinder fibered bimodule associated to $M$. We keep denoting the smooth cylinder fibered bimodule associated to an oriented cobordism $M$ by $\Omega^L_M$. 

Let $M$ and $M'$ be oriented cobordisms of dimension $n$ from an oriented manifold $\Sigma$ to another oriented manifold $\Lambda$ and from an oriented manifold $\Sigma'$ to another oriented manifold $\Lambda'$ respectively. Let $\Phi$ be an oriented equivariant diffeomorphism from $M$ to $M'$. The image $L_\Phi$ of diffeomorphism $\Phi$ under the functor of spaces of global solutions of local field theory $L$ is a smooth function from manifold $L_M$ underlying the smooth cylinder bimodule $\Omega^L_M$ associated to $M$ to the manifold $L_{M'}$ underlying the smooth cylinder bimodule $\Omega^L_{M'}$ associated to $M'$. Morphism $\Omega^L_\Phi$ of topological fibered bimodules is thus a morphism of smooth fibered bimodules from $\Omega^L_M$ to $\Omega^L_{M'}$. We write $S^L_1$ for the functor from the category of morphisms $\tilde{\mbox{\textbf{Cob}}}(n)_1$ of the symmetric monoidal double category of cornered cobordisms of dimension $n$ to the symmetric monoidal category $\tilde{\mbox{\textbf{fsGrp}}}^\ell_1$ of rigid smooth fibered bimodules over rigid fibered Lie semi-groups associating to every oriented cobordism $M$ of dimension $n$ the smooth cylinder fibered bimodule $\Omega^L_M$ associated to $M$ and to every oriented equivariant diffeomorphism $\Phi$ between oriented cobordisms of dimension $n$ the smooth morphism $E^L_\Phi$. Thus defined $S^L_1$ is easily seen to be symmetric monoidal and to make the following diagram commute.

\begin{center}

\begin{tikzpicture}
  \matrix (m) [matrix of math nodes,row sep=3em,column sep=7em,minimum width=2em]
  {
     \tilde{\mbox{\textbf{Cob}}}(n)_1& \tilde{\mbox{\textbf{fsGrp}}}^\ell_1 \\
      & \tilde{\mbox{\textbf{fsGrp}}}_1 \\};
  \path[-stealth]
    (m-1-1) edge node [above] {$S^L_1$} (m-1-2)
    (m-1-2) edge node [right] {$U_1$} (m-2-2)
    (m-1-1) edge node [below] {$E^L_1$}(m-2-2);
\end{tikzpicture}
\end{center}

\noindent We call $S^L_1$ the cylinder smooth fibered bimodule functor. We regard functors of cylinder smooth fibered bimodules as smoothings of cylinder topological fibered bimodules presented in the previous section.

\

\noindent \textsl{Smoothing cylinder topological quantum field theories}

\

\noindent We wish for pairs formed by functors of cylinder Lie semi-groups and cylinder smooth fibered bimodules associated to equivariant local field theories to form symmetric monoidal double functors extending cylinder topological quantum field theories. In order to do this we would first need to provide the pair formed by categories $\tilde{\mbox{\textbf{fsGrp}}}^\ell_0$ and $\tilde{\mbox{\textbf{fsGrp}}}^\ell_1$ with the structure of a symmetric monoidal double category thus defining a codomain double category for smooth cylinder topological quantum field theories. There is an obvious choice of concrete source and target functors on pair $\tilde{\mbox{\textbf{fsGrp}}}^\ell_0,\tilde{\mbox{\textbf{fsGrp}}}^\ell_1$ over double category $\tilde{\mbox{\textbf{fsGrp}}}$ of rigid fibered topological semi-groups. Now, given a fibered Lie semi-group $E$ the space underlying horizontal identity $i_E$ associated to $E$ inherits, from the fibered Lie semi-group structure of $E$, the structure of a smooth manifold. With this structure $i_E$ is a smooth fibered bimodule over $E$. We call $i_E$ the smooth horizontal identity fibered bimodule associated to $E$. The smooth horizontal identity fibered bimodule construction admits an extension to a symmetric monoidal functor $i$ from category $\tilde{\mbox{\textbf{fsGrp}}}^\ell_0$ to category $\tilde{\mbox{\textbf{fsGrp}}}^\ell_1$ making the following diagram commute.

\begin{center}

\begin{tikzpicture}
  \matrix (m) [matrix of math nodes,row sep=3em,column sep=7em,minimum width=2em]
  {
     \tilde{\mbox{\textbf{fsGrp}}}^\ell_0& \tilde{\mbox{\textbf{fsGrp}}}^\ell_1 \\
     \tilde{\mbox{\textbf{fsGrp}}}_0 & \tilde{\mbox{\textbf{fsGrp}}}_1 \\};
  \path[-stealth]
    (m-1-1) edge node [left] {$U_0$} (m-2-1)
            edge node [above] {$i$} (m-1-2)
    (m-1-2) edge node [right] {$U_1$} (m-2-2)
    (m-2-1) edge node [below] {$i$}(m-2-2);
\end{tikzpicture}
\end{center}

\noindent where $i$ in the bottom row of the diagram denotes the horizontal identity functor on $\tilde{\mbox{\textbf{fsGrp}}}^\ell_0$. We would now wish to define a symmetric monoidal bifunctor $\circledast$ from $\tilde{\mbox{\textbf{fsGrp}}}^\ell_0\times_{\tilde{\mbox{\textbf{fsGrp}}}^\ell_0}\tilde{\mbox{\textbf{fsGrp}}}^\ell_1$ to $\tilde{\mbox{\textbf{fsGrp}}}^\ell_1$ making the following square commute

\begin{center}

\begin{tikzpicture}
  \matrix (m) [matrix of math nodes,row sep=3em,column sep=3em,minimum width=2em]
  {
     \tilde{\mbox{\textbf{fsGrp}}}^\ell_1\times_{\tilde{\mbox{\textbf{fsGrp}}}^\ell_0}\tilde{\mbox{\textbf{fsGrp}}}^\ell_1& \tilde{\mbox{\textbf{fsGrp}}}^\ell_1 \\
     \tilde{\mbox{\textbf{fsGrp}}}_1\times_{\tilde{\mbox{\textbf{fsGrp}}}_0}\tilde{\mbox{\textbf{fsGrp}}}_1 & \tilde{\mbox{\textbf{fsGrp}}}_1 \\};
  \path[-stealth]
    (m-1-1) edge node [left] {$U_1\times_{U_0} U_1$} (m-2-1)
            edge node [above] {$\circledast$} (m-1-2)
    (m-1-2) edge node [right] {$U_1$} (m-2-2)
    (m-2-1) edge node [below] {$\circledast$}(m-2-2);
\end{tikzpicture}
\end{center}

\noindent where $\circledast$ in the bottom row of the diagram denotes the horizontal composition bifunctor on $\tilde{\mbox{\textbf{fsGrp}}}$. The following example proves that this bifunctor does not exist.

\

\noindent \textsl{Example:} Write $E$ for the following fibered Lie semi-group: We make both total and base spaces of $E$ to be equal to the real line $\mathbb{R}$. We make projection $\pi$ of $E$ to be equal to the identity diffeomorphism $id_{R}$ of $\mathbb{R}$. In this case the topological fiber product $E\times_\pi E$ of $E$ with itself relative to projection $\pi$ is equal to the diagonal $\Delta_\mathbb{R}$ on $\mathbb{R}^2$, which, considered as a submanifold of $\mathbb{R}^2$ is the fiber product of $E$ with itself relative to $\pi$ in the category of smooth manifolds and smooth functions. We make the fiber product operation on $E$ to be any of the left or right projections of $E\times_\pi E$ onto $E$. Write now $\Omega$ for the following smooth left-right fibered bimodule over $E$: We make the manifold underlying $\Omega$ to again be the real line $\mathbb{R}$. We make source and target functions $s,t$ on $\Omega$ to be the identity diffeomorphism $id_{\mathbb{R}}$ on $\mathbb{R}$ and function $x^2$ respectively. Topological fiber product $E\times_{\pi,s}\Omega$ is again equal to the diagonal $\Delta_\mathbb{R}$ on $\mathbb{R}$ and thus admits the structure of fiber product of the corresponding manifolds in the category of smooth manifolds and smooth functions. The topological fiber product $\Omega\times_{t,\pi}E$ is equal to the graph of function $x^2$. Considered as a submanifold of $\mathbb{R}^2$ $\Omega\times_{t,\pi}E$ is again equal to the fiber product relative to the corresponding diagram in the category of smooth manifolds and smooth functions. We make left action $\lambda$ of $E$ on $\Omega$ to be the fiber product operation on $E$. We make the right action $\rho$ of $E$ on $\Omega$ to be the projection on the variable axis of $\Omega\times_{t,\pi}E$ onto $\mathbb{R}$. Thus defined actions $\lambda$ and $\rho$ are easily seen to be compatible with source and target functions $s$ and $t$, associative, and to commute with each other. This and the fact that both $\lambda$ and $\rho$ are diffeomorphisms proves that with the structure described above $\Omega$ is a left-right rigid smooth fibered bimodule over $E$.
Let $\Omega'$ denote the left-right rigid fibered smooth bimodule opposite to $\Omega$, that is $\Omega'$ is the left-right rigid smooth fibered bimodule over $E$ defined as follows: We make the manifold underlying $\Omega'$ to be equal to the manifold underlying left-right fibered bimodule $\Omega$. We make source and target functions $s'$ and $t'$ on $\Omega'$ to be equal to target and source functions $s$ and $t$ on $\Omega$. Finally, we make left and right actions $\lambda'$ and $\rho'$ on $\Omega'$ to be equal to right and left actions $\rho$ and $\lambda$ on $\Omega$. With this structure it is immediate that $\Omega'$ is a left-right rigid smooth fibered bimodule over $E$. Consider now the topological horizontal composition $\Omega\circledast_{E}\Omega'$ of $\Omega$ and $\Omega'$ relative to the fibered topological semi-group underlying $E$. The space underlying $\Omega\circledast_E\Omega'$ is equal to the union of the diagonal $\Delta_\mathbb{R}$ on $\mathbb{R}^2$ and the line $x=-y$ equipped with the topology inhereted from $\mathbb{R}^2$. We conclude from this that the topological horizontal composition $\Omega\circledast_E\Omega'$ of $\Omega$ and $\Omega'$ relative to the fibered topological semi-group underlying $E$ does not admit an extension to $\tilde{\mbox{\textbf{fsGrp}}}^\ell_1$ and thus a horizontal composition operation on $\tilde{\mbox{\textbf{fsGrp}}}^\ell_0,\tilde{\mbox{\textbf{fsGrp}}}^\ell_1$ solving the compatibility diagram for forgetful functors $U_0$ and $U_1$ does not exist.

\

\noindent The next theorem provides, for each equivaraint local field theory $L$ a topological quantum field theory extending the cylinder topological quantum field theory associated to $L$ only with a codomain symmetric monoidal double category depending on field theory $L$ itself.

\

\begin{thm}
Let $n$ be a positive integer. Let $L$ be an equivariant local field theory of dimension $n$. There exists a symmetric monoidal double category $C^L$ concrete, as a symmetric monoidal double category, over $\tilde{\mbox{\textbf{fsGrp}}}$, satisfying the following two conditions:

\begin{enumerate}
\item Category of objects $C^L_0$ of double category $C^L$ is a symmetric monoidal sub-category of symmetric monoidal category $\tilde{\mbox{\textbf{fsGrp}}}^\ell_0$.
\item Category of morphisms $C^L_1$ of $C^L$ admits the structure of a symmetric monoidal category concrete over $\tilde{\mbox{\textbf{fsGrp}}}^\ell_1$.
\end{enumerate}

\noindent and a symmetric monoidal double functor $S^L$ from symmetric monoidal double category $\tilde{\mbox{\textbf{Cob}}}(n)$ of oriented cobordisms of dimension $n$ to $C^L$ such that there exists a choice of structure symmetric monoidal double functor $V$ of $C^L$ over $\tilde{\mbox{\textbf{fsGrp}}}$ and a choice of structure symmetric monoidal double functor $W$ of category of morphisms $C^L_1$ of $C^L$ over $\tilde{\mbox{\textbf{fsGrp}}}^\ell_1$ making the following two triangles commute

\begin{center}

\begin{tikzpicture}
  \matrix (m) [matrix of math nodes,row sep=3em,column sep=4em,minimum width=2em]
  {
     \tilde{\mbox{\textbf{Cob}}}(n)&C^L&C^L_1& \tilde{\mbox{\textbf{fsGrp}}}^\ell_1\\
     &\tilde{\mbox{\textbf{fsGrp}}}^\ell&&\tilde{\mbox{\textbf{fsGrp}}}_1 \\};
  \path[-stealth]
    (m-1-1) edge node [above]{$S^L$}(m-1-2)
		(m-1-2) edge node [right]{$V$}(m-2-2)
		(m-1-1) edge node [below]{$E^L$}(m-2-2)				
		
		(m-1-3) edge node [above] {$W$} (m-1-4)
    (m-1-4) edge node [right] {$U_1$} (m-2-4)
    (m-1-3) edge node [below] {$V_1$}(m-2-4);
\end{tikzpicture}
\end{center}

\end{thm}

\begin{proof}

Let $n$ be a positive integer. Let $L$ be an equivariant local field theory of dimension $n$. We wish to prove the existence of symmetric monoidal double category $C^L$, symmetric monoidal functor $W$, and symmetric monoidal double functors $V$ and $S^L$ satisfying the conditions above and making the diagrams appearing in the statement of the theorem commute. We begin by defining symmetric monoidal double category $C^L$.

We make category of objects $C^L_0$ of $C^L$ to be the image category of functor of cylinder Lie semi-groups $S_0^L$. From symmetric monoidality of functor $S_0^L$ it follows that thus defined category $C^L_0$ is a symmetric monoidal sub-category of category $\tilde{\mbox{\textbf{fsGrp}}}^\ell_1$ of reduced fibered Lie semi-groups. We now define category of morphisms $C^L_1$ of double category $C^L$. We make the collection of objects of category $C^L_1$ to be the collection of all triples of the form $(M,\Omega,\Phi)$ where $M$ is an oriented cobordism of dimension $n$ from an oriented manifold $\Sigma$ to another oriented manifold $\Sigma'$, $\Omega$ is a left-right reduced smooth fibered bimodule over cylinder fibered Lie semi-groups $E^L_\Sigma$ and $E^L_{\Sigma'}$, and $\Phi$ is an equivariant smooth isomorphism from the cylinder smooth fibered bimodule $\Omega^L_M$ associated to $M$ to $\Omega$. Given objects $(M,\Omega,\Phi)$ and $(M',\Omega',\Phi')$ in $C^L_1$ where $M$ is an oriented cobordism from oriented manifold $\Sigma$ to oriented manifold $\Lambda$ and where $M'$ is an oriented cobordism from oriented manifold $\Sigma'$ to oriented manifold $\Lambda'$ we will make collection of morphisms from triple $(M,\Omega,\Phi)$ to triple $(M'.\Omega',\Phi')$ to be the collection of all pairs of the form $(f,\Psi)$ where $f$ is an equivariant diffeomorphism of cobordisms from $M$ to $M'$ and where $\Psi$ is an isomorphism of smooth fibered bimodules from $\Omega$ to $\Omega'$ such that the following square commutes

\begin{center}

\begin{tikzpicture}
  \matrix (m) [matrix of math nodes,row sep=3em,column sep=7em,minimum width=2em]
  {
     \Omega^L_M& \Omega^L_{M'} \\
     \Omega & \Omega' \\};
  \path[-stealth]
    (m-1-1) edge node [left] {$\Phi$} (m-2-1)
            edge node [above] {$\Omega^L_f$} (m-1-2)
    (m-1-2) edge node [right] {$\Phi'$} (m-2-2)
    (m-2-1) edge node [below] {$\Psi$}(m-2-2);
\end{tikzpicture}
\end{center}

\noindent Entry-wise cartesian product makes $C^L_1$ into a symmetric monoidal category. We now provide pair $C^L$ formed by symmetric monoidal categories $C^L_0$ and $C^L_1$ with the structure of a symmetric monoidal double category. We begin by defining source and target functors $s^L$ and $t^L$ for $C^L$. Let $(M,\Omega,\Phi)$ be an object in $C^L_1$ where $M$ is an oriented cobordism of dimension $n$ from oriented manifold $\Sigma$ to oriented manifold $\Lambda$. In that case we make source $s^L(M,\Omega,\Phi)$ of triple $(M,\Omega,\Phi)$ to be cylinder fibered Lie semi-group $E^L_\Sigma$ and we make target $t^L(M,\Omega,\Phi)$ to be cylinder fibered Lie semi-group $E^L_\Lambda$. Given a morphism $(f,\Psi)$ in $C^L_1$ from an object $(M,\Omega,\Phi)$ in $C^L_1$ to an object $(M',\Omega',\Phi')$ in $C^L_1$ where $M$ is an oriented cobordism from oriented manifold $\Sigma$ to oriented manifold $\Lambda$ and where $M'$ is an oriented cobordism from oriented manifold $\Sigma'$ to oriented manifold $\Lambda'$, if we denote $f_\Sigma$ and $f_\Lambda$ the restrictions of $f$ to boundary components $\Sigma$ and $\Lambda$ respectively then, in that case we make source $s^L(f,\Psi)$ of $(f,\Psi)$ to be cylinder ismorphism of fibered Lie semi-groups $E^L_{f_\Sigma}$ associated to oriented diffeomorphism $f_\Sigma$ and we make target $t^L(f,\Psi)$ to be cylinder isomorphism $E^L_{f_\Lambda}$ associated to diffeomorphism $f_\Lambda$. It is easily seen that thus defined $s^L$ and $t^L$ define symmetric monoidal functors from category of morphisms $C^L_1$ of $C^L$ to category of objects $C^L_0$ of $C^L$. We make $s^L$ and $t^L$ to be source and target symmetric monoidal functors of $C^L$. We now define horizontal identity symmetric monoidal functor $i^L$ on $C^L$.

Let $\Sigma$ be an oriented manifold of dimension $n-1$. We will make the image $i^L_{E^L_\Sigma}$ of cylinder fibered Lie semi-group $E^L_\Sigma$ associated to $\Sigma$ under horizontal identity functor $i^L$ on $C^L$ to be the triple $(i_{E^L_\Sigma},\Sigma\times [0,1],id_{i_{E^L_\Sigma}})$ where $i_{E^L_\Sigma}$ is the smooth horizontal identity fibered bimodule associated to cylinder fibered Lie semi-group $E^L_\Sigma$ associated to $\Sigma$. Let now $\varphi$ be an oriented diffeomorphism from oriented manifold $\Sigma$ to another oriented manifold $\Lambda$ of dimension $n-1$. We make in that case horizontal identity $i^L_{E^L_\varphi}$ of cylinder morphism $E^L_\varphi$ associated to $\varphi$ to be the pair $(\varphi\times id_{[0,1]},i_{E^L_\varphi})$. Thus defined $i^L_{E^L_\varphi}$ is easily seen to be an isomorphism in category $C^L_1$. It is easy to see that functor $i^L$ defined above is symmetric monoidal. We will make $i^L$ to be the horizontal identity functor of $C^L$. Finally we define horizontal composition bifunctor $\circledast^L$ on $C^L$.

Let $(M,\Omega,\Phi)$ and $(M',\Omega',\Phi')$ be horizontally compatible objects in category $C^L_1$. Suppose $M$ is an oriented cobordism of dimension $n$ from an oriented manifold $\Sigma$ to an oriented manifold $\Sigma'$ and suppose $M'$ is an oriented cobordism of dimension $n$ from $\Sigma'$ to an oriented manifold $\Sigma''$. Suppose further that target of equivariant smooth morphism $\Phi$ is equal to source of equivariant smooth morphism $\Phi'$. Write $\varphi$ for this common automorphism of cylinder fibered Lie semi-group $E^L_{\Sigma'}$. Horizontal composition $\Phi\circledast_\varphi\Phi'$ of $\Phi$ and $\Phi'$ relative to $\varphi$ is an equivariant isomorphism of topological fibered bimodules from the horizontal composition $\Omega^L_M\circledast_{E^L_{\Sigma'}}\Omega^L_{M'}$ of horizontal compatible cylinder fibered bimodules $\Omega^L_{M}$ and $\Omega^L_{M'}$ to horizontal composition of horizontally compatible fibered bimodules $\Omega\circledast_{E^L_{\Sigma'}}\Omega'$. The associator $A_{M,M'}$ in symmetric monoidal double category $\tilde{\mbox{\textbf{fsGrp}}}$ associated to $M$ and $M'$ is an equivariant isomorphism of topological fibered bimodules from horizontal composition $\Omega^L_M\circledast_{E^L_{\Sigma'}}\Omega^L_{M'}$ to cylinder fibered bimodule $\Omega^L_{M\cup_{\Sigma'}M'}$ associated to the glued manifold $M\cup_{\Sigma'}M'$. Composition $(\Phi\circledast_\varphi\Phi')A_{M,M'}^{-1}$ is thus an equivariant isomorphism of topological fibered bimodules from cylinder fibered bimodule $\Omega^L_{M\cup_{\Sigma'}M'}$ to horizontal composition $\Omega\circledast_{E^L_{\Sigma'}}\Omega'$. We consider, on the space underlying $\Omega\circledast_{E^L_{\Sigma'}}\Omega'$, the pullback smooth structure defined by that of $\Omega^L_{M\cup_{\Sigma'}M'}$ and homeomorphism underlying $(\Phi\circledast_\varphi\Phi')A_{M,M'}^{-1}$. From the fact that components of both $(\Phi\circledast_\varphi\Phi')$ and $A_{M,M'}^{-1}$ are smooth and from the fact that fibered bimodules involved in these equivariant morphisms are smooth it follows that with the smooth structure considered above horizontal composition $\Omega\circledast_{E^L_{\Sigma'}}\Omega'$ is a smooth fibered bimodule and that with this structure equivariant isomorphism $(\Phi\circledast_\varphi\Phi')A_{M,M'}^{-1}$ is smooth. Triple $(M\cup_{\Sigma'}M',\Omega\circledast_{E^L_{\Sigma'}}\Omega',(\Phi\circledast_\varphi\Phi')A_{M,M'}^{-1})$ is thus an object in category $C^L_1$. We make this object to be the horizontal composition $(M,\Omega,\Phi)\circledast^L_{E^L_{\Sigma'}}(M',\Omega',\Phi')$ of horizontally compatible triples $(M,\Omega,\Phi)$ and $(M',\Omega',\Phi')$.  We define horizontal composition $\circledast^L$ on horizontally compatible morphisms in $C^L_1$ by entry-wise horizontal compositions in $\tilde{\mbox{\textbf{Cob}}}(n)$ and $\tilde{\mbox{\textbf{fsGrp}}}$. Thus defined horizontal composition $\circledast^L$ defines an associative symmetric monoidal bifunctor from $C^L_1\times_{C^L_0}C^L_1$ to $C^L_1$. It is obvious that with the structure defined above pair $C^L$ formed by symmetric monoidal categories $C^L_0$ and $C^L_1$ is a symmetric monoidal double category. Moreover, we have constructed $C^L_1$ in such a way that the category of objects $C^L_0$ of $C^L$ is a symmetric monoidal sub-category of category $\tilde{\mbox{\textbf{fsGrp}}}^\ell_0$. We now prove that category of morphisms $C^L_1$ of $C^L$ as defined above is concrete, as a symmetric monoidal category, over symmetric monoidal category $\tilde{\mbox{\textbf{fsGrp}}}^\ell_1$. We define a faithful symmetric monoidal functor $W$ from $C^L_1$ to $\tilde{\mbox{\textbf{fsGrp}}}^\ell_1$.

Let $(M,\Omega,\Phi)$ be an object in category $C^L$ where $M$ is an oriented cobordism of dimension $n$ from an oriented manifold $\Sigma$ to another oriented manifold $\Lambda$. We make the image $W(M,\Omega,\Phi)$ of triple $(M,\Omega,\Phi)$ under functor $W$ to be the left-right smooth fibered bimodule $\Omega$ over $E^L_\Sigma$ and $E^L_\Lambda$. Now, let $(f,\Psi)$ be a morphism in $C^L$ from object $(M,\Omega,\Phi)$ to object $(M',\Omega',\Phi')$ where $M$ is an oriented cobordism from oriented manifold $\Sigma$ to oriented manifold $\Lambda$ and $M'$ is an oriented cobordism from oriented manifold $\Sigma'$ to oriented manifild $\Lambda'$. Let $f_\Sigma$ denote the restriction of diffeomorphism $f$ to $\Sigma$ and let $f_\Lambda$ denote the restriction to $\Lambda$ of $f$. In that case we make the image $W(f,\Psi)$ of morphism $(f,\Psi)$ under functor $W$ to be smooth equivariant morphism $(E^L_{f_\Sigma},\Psi,E^L_{f_\Lambda})$. Thus defined $W(f,\Psi)$ is an equivariant smooth ismorphism from smooth fibered bimodule $\Omega$ to smooth fibered bimodule $\Omega'$. Denote by $W$ the pair formed by function associating to every object $(M,\Omega,\Phi)$ in $C^L_1$ smooth fibered bimodule $W(M,\Omega,\Phi)$ and associating to every isomorphism $(f,\Psi)$ in $C^L_1$ smooth equivariant isomorphism $W(f,\Psi)$. Thus defined $W$ is a functor from category of morphisms $C^L_1$ to category $\tilde{\mbox{\textbf{fsGrp}}}^\ell_1$. From the fact that the cylinder smooth fibered bimodule functor is symmetric monoidal it follows that functor $W$ is symmetric monoidal. Functor $W$ is clearly faithful. We conclude that category of morphisms $C^L_1$ of $C^L$ concrete, as a symmetric monoidal category, over symmetric monoidal category $\tilde{\mbox{\textbf{fsGrp}}}^\ell_1$ with functor $W$ serving as structure functor. We now define on symmetric monoidal double category $C^L$ the structure of a concrete symmetric monoidal double category over $\tilde{\mbox{\textbf{fsGrp}}}$.

We define a symmetric monoidal forgetful double functor $V$ from double category $C^L$ to double category $\tilde{\mbox{\textbf{fsGrp}}}$. We make the functor of objects $V_0$ of $V$ to be the restriction, to category of objects $C^L_0$ of $C^L$, of forgetful functor $U_0$ of $\tilde{\mbox{\textbf{fsGrp}}}^\ell_0$ onto category of objects $\tilde{\mbox{\textbf{fsGrp}}}_0$ of $\tilde{\mbox{\textbf{fsGrp}}}$. Thus defined $V_0$ is symmetric monoidal and faithful. We now make functor on morphisms $V_1$ of $V$ to be composition $U_1W$ of functor $W$ defined in previous paragraph and forgetful functor $U_1$ of $\tilde{\mbox{\textbf{fsGrp}}}^\ell_1$ over $\tilde{\mbox{\textbf{fsGrp}}}_1$. Thus defined $V_1$ is again symmetric monoidal and faithful. Moreover, it is trivial to see that pair $V$ formed by functors $V_0$ and $V_1$ forms a double functor. We conclude that symmetric monoidal double category $C^L$ is concrete, as a symmetric monoidal double category, over symmetric monoidal double cateogry $\tilde{\mbox{\textbf{fsGrp}}}$ with symmetric monoidal functor $V$ as structure symmetric monoidal double functor. Moreover, from the way it was defined functor of morphisms $V_1$ of double functor $V$ and structure functor $W$ of category of morphisms $C^L_1$ of $C^L$ over $\tilde{\mbox{\textbf{fsGrp}}}^\ell_1$ make the following diagram commutative.

\begin{center}

\begin{tikzpicture}
  \matrix (m) [matrix of math nodes,row sep=3em,column sep=7em,minimum width=2em]
  {
     C^L_1&\tilde{\mbox{\textbf{fsGrp}}}^\ell_1\\
          &\tilde{\mbox{\textbf{fsGrp}}}_1 \\};
  \path[-stealth]
    (m-1-1) edge node [above] {$W$} (m-1-2)
    (m-1-2) edge node [right] {$U_1$} (m-2-2)
    (m-1-1) edge node [below] {$V_1$}(m-2-2);
\end{tikzpicture}
\end{center}

\noindent We now define symmetric monoidal double functor $S^L$ from double category $\tilde{\mbox{\textbf{Cob}}}(n)$ to double category $C^L$ satisfying the conditions prescribed in the statement of the theorem. We make functor on objects $S^L_0$ of $S^L$ to be the cylinder Lie semi-group symmetric monoidal functor associated to $L$. We now define functor on morphisms $S^L_1$ of $S^L$. Let $\Sigma$ and $\Lambda$ be oriented manifolds of dimension $n-1$. Let $M$ be an orented cobordism from $\Sigma$ to $\Lambda$. We make the image $S^L_1(M)$ of $M$ under functor $S^L_1$ to be the triple $(M,\Omega^L_M,id_{\Omega^L_M})$ formed by cobordism $M$, left-right cylinder smooth fibered bimodule $\Omega^L_M$, and the identity isomorphism $id_{\Omega^L_M}$. Thus defined $S^L_1(M)$ is an object in category of morphisms $C^L_1$ of $C^L$. Let now $\Sigma'$ and $\Lambda'$ be another pair of oriented manifolds of dimension $n-1$ and let $M'$ be an oriented cobordism from $\Sigma'$ to $\Lambda'$. Let $\varphi$ be an equivariant oriented diffeomorphism from $M$ to $M'$. In that case we make the image $S^L_1(\varphi)$ of $\varphi$ under functor $S^L_1$ to be the pair $(\varphi,\Omega^L_\varphi)$ formed by diffeomorphism $\varphi$ and cylinder smooth equivariant isomorphism $\Omega^L_\varphi$ associated to $\varphi$. Thus defined $S^L_1$ is a symmetric monoidal functor from category of morphisms $\tilde{\mbox{\textbf{Cob}}}_1$ of $\tilde{\mbox{\textbf{Cob}}}$ to cateogry of morphisms $C^L_1$ of $C^L$. We now prove that the pair $C^L$ formed by symmetric monoidal functors is a double functor.

\end{proof}

\section{Recovering symplectic structure}

\noindent In this section we extend the smooth cylinder topological quantum field theory construction to the construction of topological quantum field theories associated to equivariant local field theories implementing the symplectic structure of spaces of germs of local solutions. The presentation fo this extension procedure will be analogous to that of the smoothing procedure for topological cylinder quantum field theories presented in the previous section. We begin with the definition of fibered symplectic semi-group and Lagrangian fibered bimodules.

\

\noindent\textit{Fibered symplectic semi-groups and their bimodules}

\

\noindent We will say that a fibered Lie semi-group $E$ is symplectic whenever the base manifold of $E$ is equipped with the structure of a symplectic manifold. We will say that a fibered symplectic semi-group $E$ is rigid when the fibered Lie semi-group underlying $E$ is rigid. Given a morphism of fibered Lie semi-groups $\Phi$ from the underlying fibered Lie semi-group of a fibered symplectic semi-group $E$ to the underlying fibered Lie semi-group of another fibered symplectic semi-group $E'$ is a morphism of fibered symplectic semi-groups if base space function of $\Phi$ is a symplectomorphism. We write \textbf{fsGrp}$^s_0$ for the category of fibered symplectic semi-groups and their morphisms. Cartesian product provides \textbf{fsGrp}$^s_0$ with the structure of a symmetric monoidal category concrete over \textbf{fsGrp}$^\ell_0$. We write $\tilde{\mbox{\textbf{fsGrp}}}^s_0$ for the full symmetric monoidal sub-category of \textbf{fsGrp}$^\ell_0$ generated by rigid fibered symplectic semi-groups. Thus defined $\tilde{\mbox{\textbf{fsGrp}}}^s_0$ is concrete, as a symmetric monoidal category, over $\tilde{\mbox{\textbf{fsGrp}}}^\ell_0$. We write $\tilde{U}_0$ for the obvious forgetful functor of $\tilde{\mbox{\textbf{fsGrp}}}^s_0$ over $\tilde{\mbox{\textbf{fsGrp}}}^\ell_0$.

Let $E$ and $E'$ be fibered symplectic semi-groups with base symplectic manifolds $X$ and $Y$ respectively. Let $\Omega$ be a left-right smooth fibered bimodule over $E$ and $E'$ with source and target functions $s$ and $t$. We will understand for a structure of Lagrangian fibered bimodule on $\Omega$ a smooth function $\xi$ from $\Omega$ to the product $X\times Y$ such that source and target functions $s$ and $t$ on $\Omega$ are implemented by $\xi$, i.e. such that function $s$ equals composition $p_X\xi$ and target function $t$ equals composition $p_Y\xi$ where $p_X$ and $p_Y$ denote the projections on $X$ and $Y$ respectively of product $X\times Y$ and such that the image $im\xi$ of $\xi$ is a Lagrangian sub-manifold of $X\times -Y$. We will say that a smooth fibered bimodule together with a structure of Lagrangian fibered bimodule is a Lagrangian fibered bimodule. We will say that a Lagrangian fibered bimodule $\Omega$ is reduced whenever the smooth fibered bimodule underlying $\Omega$ is reduced.

Given an equivariant smooth morphism $\Phi$ from the smooth fibered bimodule underlying a Lagrangian fibered bimodule $\Omega$ to the smooth fibered bimodule underlying another Lagrangian fibered bimodule $\Omega'$ we say that $\Phi$ is an equivariant morphism of Lagrangian fibered bimodules from $\Omega$ to $\Omega'$ if $\Phi$ intertwines Lagrangian structures in $\Omega$ and $\Omega'$. Precisely, if $\Omega$ is a Lagrangian fibered bimodule over fibered symplectic semi-groups $E$ and $F$ with base symplectic manifolds $X$ and $Y$ respectively and Lagrangian structure $\xi$, $\Omega'$ is a fibered bimodule over fibered symplectic semi-groups $E'$ and $F'$ with base symplectic manifolds $X'$ and $Y'$ and Lagrangian structure $\xi'$, and if functions on base spaces of source and target fibered symplectic semi-group morphisms of $\Phi$ are $\varphi$ and $\psi$ respectively then $\Phi$ is Lagrangian if the following diagram commutes

\begin{center}

\begin{tikzpicture}
  \matrix (m) [matrix of math nodes,row sep=3em,column sep=7em,minimum width=2em]
  {
     \Omega&\Omega' \\
     E\times F&E'\times F'\\};
  \path[-stealth]
    (m-1-1) edge node [left] {$\xi$} (m-2-1)
            edge node [above] {$\Phi$} (m-1-2)
    (m-1-2) edge node [right] {$\xi'$} (m-2-2)
    (m-2-1) edge node [below] {$\varphi\times\psi$}(m-2-2);
\end{tikzpicture}
\end{center}

\noindent We will write \textbf{fsGrp}$^s_1$ for the symmetric monoidal category of Lagrangian fibered bimodules over fibered symplectic semi-groups and we will write $\tilde{\mbox{\textbf{fsGrp}}}^s_1$ for the full symmetric monoidal sub-category of \textbf{fsGrp}$^s_1$ generated by rigid Lagrangian fibered bimodules over rigid fibered symplectic semi-groups. Thus defined \textbf{fsGrp}$^s_1$ and $\tilde{\mbox{\textbf{fsGrp}}}^s_1$ are concrete, as symmetric monoidal categories over \textbf{fsGrp}$^\ell_1$ and $\tilde{\mbox{\textbf{fsGrp}}}^\ell_1$ respectively. We will write $\tilde{U}_1$ for the obvious forgetful functor in both cases.

\

\noindent Let $n$ be a positive integer. Let $L$ be an equivariant local field theory of dimension $n$. Given an oriented manifold of dimension $n-1$ field theory $L$ associates to the base manifold $L_\Sigma$ of cylinder fibered Lie semi-group $E^L_\Sigma$ associated to $\Sigma$ the structure of a symplectic manifold and thus associates to $E^L_\Sigma$ the structure of a rigid fibered symplectic semi-group. We call $E^L_\Sigma$ with this structure the cylinder fibered symplectic semi-group associated to $\Sigma$. The cylinder fibered symplectic semi-group construction extends to a symmetric monoidal functor from category of objects $\tilde{\mbox{\textbf{Cob}}}(n)_0$ of $\tilde{\mbox{\textbf{Cob}}}(n)$ to category $\tilde{\mbox{\textbf{fsGrp}}}^s_0$ of rigid fibered symplectic semi-groups. We write $\tilde{S}^L_0$ for this functor and we call $\tilde{S}^L_0$ the functor of cylinder fibered symplectic semi-groups associated to $L$. Thus defined $\tilde{S}^L_0$ makes the following diagram commute

\begin{center}

\begin{tikzpicture}
  \matrix (m) [matrix of math nodes,row sep=3em,column sep=7em,minimum width=2em]
  {
     \tilde{\mbox{\textbf{Cob}}}(n)_0& \tilde{\mbox{\textbf{fsGrp}}}^s_0 \\
     \tilde{\mbox{\textbf{fsGrp}}}_0 & \tilde{\mbox{\textbf{fsGrp}}}^\ell_0 \\};
  \path[-stealth]
    (m-1-1) edge node [left] {$E^L_0$} (m-2-1)
            edge node [above] {$\tilde{S}^L_0$} (m-1-2)
    (m-1-2) edge node [right] {$\tilde{U}_0$} (m-2-2)
    (m-2-2) edge node [below] {$V_0$}(m-2-1);
\end{tikzpicture}
\end{center}

\noindent and thus makes the following diagram commute

\begin{center}

\begin{tikzpicture}
  \matrix (m) [matrix of math nodes,row sep=3em,column sep=7em,minimum width=2em]
  {
     \tilde{\mbox{\textbf{Cob}}}(n)_0& \tilde{\mbox{\textbf{fsGrp}}}^s_0 \\
     & \tilde{\mbox{\textbf{fsGrp}}}_1 \\};
  \path[-stealth]
    (m-1-1) edge node [above] {$\tilde{S}^L_0$} (m-1-2)
    (m-1-2) edge node [right] {$\tilde{U}_0$} (m-2-2)
    (m-1-1) edge node [below] {$S^L_0$}(m-2-2);
\end{tikzpicture}
\end{center}

\noindent Given an oriented cobordism $M$ of dimension $n$ from an oriented manifold $\Sigma$ to another oriented manifold $\Lambda$ field theory $L$ associates to the manifold $L_M$ underlying cylinder smooth fibered bimodule $\Omega^L_M$ a smooth function, the restriction transformation associated to $M$, from $L_M$ to the product $L_\Sigma\times L_\Lambda$ of base symplectic manifolds $L_\Sigma$ and $L_\Lambda$ of cylinder fibered symplectic semi-groups $E^L_\Sigma$ and $E^L_\Lambda$ associated to $\Sigma$ and $\Lambda$. Function $r_M$ implements source and target functions on $\Omega^L_M$ and the image $im r_M$ is a symplectic sub-manifold of $L_\Sigma\times -L_\Lambda$. Thus defined $r_M$ thus provides $\Omega^L_M$ with the structure of a Lagrangian fibered bimodule over cylinder fibered symplectic semi-groups $E^L_\Sigma$ and $E^L_\Lambda$. We call $\Omega^L_M$ with this structure the cylinder Lagrangian fibered bimodule associated to $M$. As in the case of the cylinder smooth fibered bimodule construction the cylinder Lagrangian fibered bimodule construction extends to a symmetric monoidal functor from category of morphisms $\tilde{\mbox{\textbf{Cob}}}(n)_0$ of $\tilde{\mbox{\textbf{Cob}}}(n)$ to cateogory  $\tilde{\mbox{\textbf{fsGrp}}}^s_1$ of rigid Lagrangian fibered bimodules. We write $\tilde{S}^L_1$ for this functor. We call $\tilde{S}^L_1$ the functor of cylinder Lagrangian fibered bimodules associated to $L$. Functor $\tilde{S}^L_1$ makes the following diagram commute.

\begin{center}

\begin{tikzpicture}
  \matrix (m) [matrix of math nodes,row sep=3em,column sep=7em,minimum width=2em]
  {
     \tilde{\mbox{\textbf{Cob}}}(n)_1& \tilde{\mbox{\textbf{fsGrp}}}^s_1 \\
     \tilde{\mbox{\textbf{fsGrp}}}_1 & \tilde{\mbox{\textbf{fsGrp}}}^\ell_1 \\};
  \path[-stealth]
    (m-1-1) edge node [left] {$E^L_1$} (m-2-1)
            edge node [above] {$\tilde{S}^L_1$} (m-1-2)
    (m-1-2) edge node [right] {$\tilde{U}_1$} (m-2-2)
    (m-2-2) edge node [below] {$V_1$}(m-2-1);
\end{tikzpicture}
\end{center}

\noindent and thus makes the following triangle commute

\begin{center}

\begin{tikzpicture}
  \matrix (m) [matrix of math nodes,row sep=3em,column sep=7em,minimum width=2em]
  {
     \tilde{\mbox{\textbf{Cob}}}(n)_0& \tilde{\mbox{\textbf{fsGrp}}}^s_0 \\
     & \tilde{\mbox{\textbf{fsGrp}}}_1 \\};
  \path[-stealth]
    (m-1-1) edge node [above] {$\tilde{S}^L_1$} (m-1-2)
    (m-1-2) edge node [right] {$\tilde{U}_1$} (m-2-2)
    (m-1-1) edge node [below] {$S^L_1$}(m-2-2);
\end{tikzpicture}
\end{center}

\

\noindent \textit{Symplectic cylinder topological quantum field theories}

\

\noindent We now present a version of theorem 5.1 now providing an extension of the smooth cylinder topological quantum field theory construction taking symplectic structures into consideration. Its proof is analogous, up to certain details, to the proof of theorem 5.1.

\begin{thm}
Let $n$ be a positive integer. Let $L$ be an equivaraint field theory of dimension $n$. There exists a symmetric monoidal double category $\tilde{C}^L$, concrete as a symmetric monoidal double category, over $\tilde{\mbox{\textbf{fsGrp}}}$ such that

\begin{enumerate}
\item Category of objects $\tilde{C}^L_0$ of double category $\tilde{C}^L$ is a symmetric monoidal sub-category of symmetric monoidal category $\tilde{\mbox{\textbf{fsGrp}}}^s_0$.
\item Category of morphisms $\tilde{C}^L_1$ admits the structure of a concrete symmetric monoidal category over $\tilde{\mbox{\textbf{fsGrp}}}^s_1$
\end{enumerate}

\noindent and a symmetric monoidal double functor $\tilde{S}^L$ from symmetric monoidal double category $\tilde{\mbox{\textbf{Cob}}}(n)$ of cobordisms of dimension $n$ to $\tilde{C}^L$ such that there exist choices of structure symmetric monoidal functor $\tilde{W}$ of $\tilde{C}^L_1$ over $\tilde{\mbox{\textbf{fsGrp}}}^s_1$ and of structure symmetric monoidal double functor $\tilde{U}$ of $\tilde{C}^L$ over $\tilde{\mbox{\textbf{fsGrp}}}^s$ such that the following two diagrams commute

\begin{center}

\begin{tikzpicture}
  \matrix (m) [matrix of math nodes,row sep=3em,column sep=4em,minimum width=2em]
  {
     \tilde{\mbox{\textbf{Cob}}}(n)&\tilde{C}^L&\tilde{C}^L_1& \tilde{\mbox{\textbf{fsGrp}}}^s_1\\
     &\tilde{\mbox{\textbf{fsGrp}}}&&\tilde{\mbox{\textbf{fsGrp}}}_1 \\};
  \path[-stealth]
    (m-1-1) edge node [above]{$\tilde{S}^L$}(m-1-2)
		(m-1-2) edge node [right]{$\tilde{V}$}(m-2-2)
		(m-1-1) edge node [below]{$E^L$}(m-2-2)				
		
		(m-1-3) edge node [above] {$\tilde{W}$} (m-1-4)
    (m-1-4) edge node [right] {$\tilde{U}_1$} (m-2-4)
    (m-1-3) edge node [below] {$\tilde{V}_1$}(m-2-4);
\end{tikzpicture}
\end{center}

\end{thm}

\begin{proof}
Let $n$ be a positive integer. Let $L$ be an equivariant local field theory of dimension $n$. We wish to prove the existence of a symmetric monoidal double category $\tilde{C}^L$ and a symmetric monoidal double functor $\tilde{S}^L$ from double category $\tilde{\mbox{\textbf{Cob}}}(n)$ to $\tilde{C}^L$ satisfying the conditions of the theorem.

We begin by constructing symmetric monoidal double category $\tilde{C}^L$. The construction of double cateogry $\tilde{C}^L$ will be analogous to the construction of double category $C^L$ in theorem 5.1. We make category of objects $\tilde{C}^L_0$ of $\tilde{C}^L$ to be the image symmetric monoidal category of cylinder fibered symplectic semi-group functor $\tilde{S}^L_0$. Thus defined $\tilde{C}^L_0$ is a symmetric monoidal sub-category of $\tilde{\mbox{\textbf{fsGrp}}}^s_0$. We now define category of morphisms $\tilde{C}^L_1$ of $\tilde{C}^L$. We make collection of objects of of $\tilde{C}^L_1$ to be the collection of all triples of the form $(M,\Omega,\Phi)$ where $M$ is again an oriented cobordism of dimension $n$ from an oriented manifold $\Sigma$ to another oriented manifold $\Lambda$, $\Omega$ is now a Lagrangian fibered bimodule over $E^L_\Sigma$ and $E^L_\Lambda$ and $\Phi$ is an equivariant ismorphism of Lagrangian fibered bimodules from cylinder Lagrangian fibered bimodule $\Omega^L_M$ associated to $M$ to $\Omega$. The definition of colection of morphisms of $\tilde{C}^L_1$ is analogous to the definition of collection of morphisms of category $C^L_1$ associated to $L$. Thus defined $\tilde{C}^L_1$ is symmetric monoidal and clearly concrete over $\tilde{\mbox{\textbf{fsGrp}}}^s_1$. We now define Lagrangian horizontal identity functor $\tilde{i}^L$ for $\tilde{C}^L$.

Let $\Sigma$ be an oriented manifold of dimension $n-1$. In that case restriction transformation $r_{\Sigma\times [0,1]}$ associated to the cylinder $\Sigma\times [0,1]$ on $\Sigma$ is a smooth function from underlying manifold $L_{\Sigma\times [0,1]}$ of smooth horizontal identity $i_{\Sigma\times [0,1]}$ of cylinder fibered Lie semi-group $E^L_\Sigma$ associated to $\Sigma$ to the product $L_\Sigma^2$ of the base symplectic manifold $L_\Sigma$ of $E^L_\Sigma$ with itself. Function $r_{\Sigma\times [0,1]}$ implements source and target functions on $i_{E^L_\Sigma}$ and the image $imr_{\Sigma\times [0,1]}$ is equal to the diagonal sub-manifold $\Delta_{L_\Sigma}$ of $L_{\Sigma}^2$ and thus is a Lagrangian submanifold of $L_\Sigma\times -L_\Sigma$. Restriction transformation $r_{\Sigma\times [0,1]}$ thus provides smooth horizontal identity fibered bimodule $i_{E^L_\Sigma}$ associated to $E^L_\Sigma$ with the structure of a Lagrangian fibered bimodule over cylinder fibered symplectic semi-group $E^L_\Sigma$. We make the image $\tilde{i}^L_{E^L_\Sigma}$ of $E^L_\Sigma$ under horizontal identity functor $\tilde{i}^L$ associated to $L$ to be the triple $(\Sigma\times [0,1],i_{E^L_\Sigma},id_{i_{E^L_\Sigma}})$ where we equip smooth horizontal identity $i_{E^L_\Sigma}$ with the Lagrangian structure described above. We make horizontal identity functor $\tilde{i}^L$ on $\tilde{C}^L$ act in the obvious manner on morphisms of $\tilde{C}^L_0$. Thus defined $\tilde{i}^L$ is a symmetric monoidal functor from category of objects $\tilde{C}^L_0$ of $\tilde{C}^L$ to category of morphisms $\tilde{C}^L_1$ of $\tilde{C}^L$. We now define horizontal composition bifunctor $\tilde{\circledast}^L$ on $\tilde{C}^L$.

Let $\Sigma,\Sigma'$ and $\Sigma''$ be oriented manifolds of dimension $n-1$. Let $M$ be an oriented cobordism from $\Sigma$ to $\Sigma'$ and let $M'$ be an oriented cobordism from $\Sigma'$ to $\Sigma''$. Associator $A_{M,M'}$ associated to $M$ and $M'$ by cylinder topological semi-group double functor $E^L$ defines, by the gluing transformation axiom in the definition of $L$, a smooth equivariant isomorphism of smooth fibered bimodules from cylinder smooth fibered bimodule $\Omega^L_{M\cup_{\Sigma'}M'}$ associated to the glued manifold $M\cup_{\Sigma'}M'$ to the horizontal composition $\Omega^L_M\circledast_{E^L_{\Sigma'}}\Omega^L_{M'}$ of cylinder smooth fibered bimodules $\Omega^L_M$ and $\Omega^L_{M'}$ equipped with the smooth structure described in the proof of theorem ?. Let $\xi_{M,M'}$ denote the pullback of Lagrangian structure $r_{M\cup_{\Sigma'}M'}$ on $\Omega^L_{M\cup_{\Sigma'}M'}$ relative to $A_{M,M'}$. Thus defined $\xi_{M,M'}$ is a smooth function from manifold $L_M\times_{L_{\Sigma'}}L_{M'}$ underlying $\Omega^L_M\circledast_{E^L_{\Sigma'}}\Omega^L_{M'}$ to the product $L_\Sigma\times L_{\Sigma''}$ of base manifolds of $E^L_\Sigma$ and $E^L_{\Sigma''}$ making the following diagram commute

\[r_{M\cup_{\Sigma'}M'}A_{M,M'}=\xi_{M,M'}\]

\noindent From the commutativity of the diagram above it follows that the image $im \xi_{M,M'}$ of $\xi_{M,M'}$ is a Lagrangian sub-manifold of $L_\Sigma\times -L_{\Sigma''}$. Moreover, from the fact that $A_{M,M'}$ is an equivariant isomorphism of fibered bimodules it follows that $\xi_{M,M'}$ implements source and target functions of $\Omega^L_M\circledast_{E^L_{\Sigma'}}\Omega^L_{M'}$. We conclude that thus defined $\xi_{M,M'}$ is a Lagrangian structure on $\Omega^L_M\circledast_{E^L_{\Sigma'}}\Omega^L_{M'}$. Let now $\Omega$ be a Lagrangian fibered bimodule over $E^L_\Sigma$ and $E^L_{\Sigma'}$ and let $\Omega'$ be a Lagrangian fibered bimodule over $E^L_{\Sigma'}$ and $E^L_{\Sigma''}$. Let $\Phi$ be an equivariant isomorphism of Lagrangian fibered bimodules from $\Omega^L_M$ to $\Omega$ and let $\Phi'$ be an equivariant isomorphism of Lagrangian fibered bimodules from $\Omega^L_{M'}$ to $\Omega'$. In that case we make the horizontal composition $(M,\Omega,\Phi)\tilde{\circledast}^L(M',\Omega',\Phi')$ of objects $(M,\Omega,\Phi)$ and $(M',\Omega',\Phi')$ in $\tilde{C}^L_1$ to be triple $(M\cup_{\Sigma'}M',\Omega\circledast_{E^L_{\Sigma'}} \Omega',\Phi\circledast\Psi)$ where $\Omega\circledast_{E^L_{\Sigma'}}\Omega'$ is equipped with the Lagrangian structure obtained by pulling back Lagrangian structure $\xi_{M,M'}$ on $\Omega^L_M\circledast_{E^L_{\Sigma'}}\Omega^L_{M'}$ by $\Phi\circledast\Phi'$. We define horizontal composition of equivariant morphisms of Lagrangian equivariant morphisms in the obvious manner thus defining a symmetric monoidal bifunctor $\tilde{\circledast}^L$ from category $\tilde{C}^L_1\times_{\tilde{C}^L_0}\tilde{C}^L_1$ to $\tilde{C}^L_1$. Arguments analogous as the ones employed in the proof of theorem 5.1 prove that with this structure $\tilde{C}^L$ is a symmetric monoidal double cateogry satisfying the conditions prescribed in the theorem. We now define symmetric monoidal double functor $\tilde{S}^L$ satisfying the conditions of the theorem.

We make functor on objects $\tilde{S}^L_0$ of $\tilde{S}^L$ to be the functor of cylinder fibered symplectic semi-groups. We now define functor of morphisms $\tilde{S}^L_1$ of $\tilde{S}^L$. Let $\Sigma$ and $\Sigma'$ be oriented manifolds of dimension $n-1$. Let $M$ be an oriented cobordism from $\Sigma$ to $\Sigma'$. In that case we make the image of $M$ under $\tilde{S}^L_1$ to be the triple $(M,\Omega^L_M,id_{\Omega^L_M})$ where $\Omega^L_M$ is the cylinder Lagrangian fibered bimodule associated to $M$. We define $\tilde{S}^L_1$ on morphisms in the obvious manner. Again arguments analogous as the ones presented in the proof of theorem ? prove that thus defined $\tilde{S}^L$ is a symmetric monoidal double functor from symmetric monoidal double category $\tilde{\mbox{\textbf{Cob}}}(n)$ to symmetric monodial double cateogry $\tilde{C}^L$ satisfying the conditions prescribed in the statement of the theorem.

\end{proof}

\noindent Given a local equivariant field theory $L$ we call topological quantum field theory $\tilde{S}^L$ defined in the theorem above the symplectic cylinder topological quantum field theory associated to $L$.

\section{Bibliography}

\noindent [1] Oeckl R., A local and operational framework for the foundations of physics. arXiv:1610.09052

\

\noindent [2] Oeckl E., A positive formalism for quantum theory in the general boundary formulation. Found. Phys. 43 (2013) 1206-1232. 

\

\noindent [3] Oeckl R., Observables in the General Boundary Formulation. Quantum Field Theory and Gravity (Regensburg, 2010), 2012, pp. 137-156.

\

\noindent [4] Grandis M, Paré R., Limits in double categories Cahiers de topologie et géométrie différentielle catégoriques, tome 40, no 3 (1999), p. 162-220

\

\noindent [5] Paré R., Yoneda Theory for Double Cateogries, Theory and Applications of Categories, Vol. 25, No. 17, 2011, pp. 436–489.

\

\noindent [6] Lurie J. On the Classification of Topological Quantum Field Theory.

\

\noindent [7] Carqueville N, Runkel I., Introductory Lectures on Topological Quantum Field Theory arxiv.org/pdf/1705.05734.pdf

\

\noindent [8] Barret J. W., B. W. Westbury, Invariants of Piecewise-Linear 3-Manifolds. Trans.Amer.Math.Soc. 348 (1996) 3997-4022.

\

\noindent [9] Reshetikhin N., Viro O. Ribbon graphs and their invariants derived from quantum groups,
Comm. Math. Phys. Volume 127, Number 1 (1990), 1-26.

\

\noindent [10] Kirillov Jr. A., Balsam B. Turaev-Viro invariants as an extended TQFT. arXiv:1004.1533.

\

\noindent [11] Tsumura Y. A 2-categorical extension of the Reshetikhin-Turaev theory. arXiv:1309.3630.

\

\noindent [12] Gelca R. Topological quantum field theory with corners based on the Kauffman bracket. Commentarii Mathematici Helvetici September 1997, Volume 72, Issue 2, pp 216–243

\

\noindent [13] Atiyah M. Topological Quantum Field Theory. Pubicationes MAthematiques de l'IHES 64 (1988) 175-186.

\

\noindent [14] Al Quasmi K., Stokman J.V. The skein category of the annulus. arXiv:1710.04058.

\

\noindent [15] Brundan J. Representations of the oriented skein category. arXiv:1712.08953.

\

\noindent [16] Joyal A., Street R. Braided Monoidal Categories. Macquarie Mathematics Reports. 

\

\noindent [19] Chen J. The Temperley-Lieb Categories and Skein Modules. arxiv.org/pdf/1502.06845.

\

\noindent [20] Abramsky S. Temperley-Lieb Algebra: From Knot Theory to Logic and Computation via Quantum Mechanics.

\

\noindent [21] Nesheyev S., Yamashita M. A Few Remarks on the Tube Algebra of a Monoidal Category. arxiv.org/pdf/1511.06332.pdf

\

\noindent [22] Evans D. E., Kawahigashi Y. On Ocneanu’s theory of asymptotic inclusions for subfactors, topological quantum field theories and quantum doubles, Internat. J. Math. 6 (1995), no. 2, 205–228.

\

\noindent [23] Walker K. Morrison S. The Blob Complex. Geom. Topol. 16 (2012) 1481-1607.

\

\noindent [24] Walker K. TQFT's. http://canyon23.net/math/tc.pdf.

\

\noindent [25] Abrams L. Two-dimensional Topological Quantum Field Theories and Frobenius Algebras. 	J. Knot Theory Ramifications.

\

\noindent [26] Gwilliam K., Costello O. Factorization Algebras in Quantum Field Theory

\

\noindent [27] Williams B. R. The Virasoro vertex algebra and factorization algebras on Riemann surfaces. 	arXiv:1603.02349.

\end{document}